\documentclass{amsart}

\usepackage[all]{xypic}

\usepackage{tikz}
\usetikzlibrary{arrows}
\usetikzlibrary{decorations.markings}

\usetikzlibrary{arrows}
\usepackage{epsf}
\usepackage{verbatim} 
\usepackage{amsmath}
\usepackage{amsfonts}
\usepackage{amssymb}
\usepackage{mathrsfs}
\usepackage{amsthm}
\usepackage{newlfont}
 \newtheorem{thm}{Theorem}[section]
 \newtheorem{cor}[thm]{Corollary}
 \newtheorem{lem}[thm]{Lemma}
 \newtheorem{prop}[thm]{Proposition}
 \theoremstyle{definition}
 \newtheorem{defn}[thm]{Definition}
 \theoremstyle{remark}
 \newtheorem{rem}[thm]{Remark}
 \theoremstyle{definition}
 \newtheorem{conj}[thm]{Conjecture}
 
\theoremstyle{definition}
 \newtheorem{example}[thm]{Example}

 \theoremstyle{definition}
 \newtheorem{notn}[thm]{Notation}
 \numberwithin{equation}{section}

 \newcommand{\PGL}{\mathrm{PGL}}
 \newcommand{\SL}{\mathrm{SL}}

 \newcommand{\Type}{\mathrm{Type}}

 \newcommand{\Ver}{\mathrm{Ver}}
 \newcommand{\St}{\mathrm{St}}
 \newcommand{\Lk}{\mathrm{Lk}}
 
\newcommand{\sC}{\mathscr{C}}

\newcommand{\sZ}{\mathscr{Z}}   
 
 \newcommand{\fT}{\mathfrak T}

 \newcommand{\fB}{\mathfrak B}
 \newcommand{\cO}{\mathcal{O}}
 \newcommand{\cV}{\mathcal{V}}

 \newcommand{\cG}{\mathcal{G}}

 \renewcommand{\cD}{\mathcal{D}}

 \newcommand{\cF}{\mathcal{F}}

 \renewcommand{\cL}{\mathcal{L}}
 \renewcommand{\cH}{\mathcal{H}}
 
 \newcommand{\R}{\mathbb{R}}
 
 \newcommand{\F}{\mathbb{F}}
 \newcommand{\Q}{\mathbb{Q}}
 \newcommand{\Z}{\mathbb{Z}}

 \newcommand{\eps}{\varepsilon}

 \newcommand{\G}{\Gamma}
 
 \newcommand{\La}{\Lambda}
 \newcommand{\la}{\lambda}

 \newcommand{\harm}[1]{\cH^{#1}}
 
 \newcommand{\gc}[2]{\left[\begin{matrix} #1 \\ #2\end{matrix} \right]_{\mkern-2.8mu q}}

\begin{document}

\title[On Garland's vanishing theorem for $\SL(n)$]
{On Garland's vanishing theorem for $\mathbf{SL(n)}$}

\author{Mihran Papikian}

\address{Department of Mathematics, Pennsylvania State University, University Park, PA 16802, U.S.A.}
\email{papikian@psu.edu}

\thanks{The author's research was partially supported by a grant from the Simons Foundation (245676) and 
the NSA Young Investigator Grant (H98230-15-1-0008).} 



\date{}

\begin{abstract} This is an expository paper on Garland's vanishing theorem 
specialized to the case when the linear algebraic group is $\SL_n$. 
Garland's theorem can be stated as a vanishing of the cohomology groups 
of certain finite simplicial complexes. 
The method of the proof is quite interesting on its own. It relates the vanishing of cohomology 
 to the assertion that the minimal positive eigenvalue of a certain combinatorial laplacian is sufficiently large. 
Since the 1970's, this idea has found applications in a variety of problems 
in representation theory, group theory, and combinatorics, so the paper might be 
of interest to a wide audience. The paper is intended for non-specialists and graduate students. 
\end{abstract}

\maketitle

\section{Introduction}

\subsection{Statement of the theorem}
This is an expository paper on Howard Garland's work \cite{Garland} specialized to the case when the
linear algebraic group is $\SL_n$. Reading \cite{Garland} requires
knowledge of the theory of buildings. On the other hand, the
ideas in \cite{Garland}, in their essence, are combinatorial. 
Since the relevant Bruhat-Tits building in the $\SL_n$ case has a simple
description in terms of lattices in a finite dimensional vector space, 
one can give a proof of Garland's vanishing 
theorem which requires from the reader only familiarity with linear
algebra and some group theory. 
There is already an excellent Bourbaki Expos\'e by Borel
\cite{Borel} on Garland's work. The main difference of our
exposition, besides the more detailed proofs, is the absence of any
references to the theory of buildings, which makes the article
completely self-contained. Since Garland's result applies to quite
general discrete subgroups of $p$-adic groups, to give a full
account of his work one cannot avoid a discourse on the 
theory of buildings and representation theory. Other expositions 
of Garland's method and its generalizations can be found in \cite{BS} and \cite{ABM}; 
these papers also give nice applications of Garland's method to group theory and combinatorics,  
although they do not give a complete proof of Theorem \ref{thm1.1}.  

Let $K$ be a non-archimedean local field with residue field of order $q$; 
such a field is either isomorphic to a finite extension of the $p$-adic numbers $\Q_p$, or 
to the field of formal Laurent series $\F_q(\!(T)\!)$ over the finite field $\F_q$ with $q$ elements.  
Let $\G$ be a discrete subgroup of the topological group $\SL_n(K)$, $n\geq 2$.  
There is an infinite contractible $(n-1)$-dimensional simplicial complex $\fB$, the \textit{Bruhat-Tits
building of $\SL_{n}(K)$}, on which $\G$ naturally acts. 
The complex $\fB$ can be described in terms of lattices in the vector space $K^n$. 
One way to formulate the main result of \cite{Garland} for $\SL_n$ is as follows: 

\begin{thm}\label{thm1.1} Assume the quotient $\fB/\G$ is a finite complex. 
There is a constant $q(n)$ depending only on $n$ such that if $q\geq q(n)$ 
then for all $0<i<n-1$ the simplicial cohomology groups $H^i(\fB/\G, \R)$ are zero. 
\end{thm}

Such vanishing theorems were originally conjectured by Serre \cite{Serre}. 
We will prove Theorem \ref{thm1.1} in Section \ref{sBTB} under a mild assumption on $\G$. 
In the same section we also explain how central division algebras 
naturally give rise to such 
groups. 
We should mention that the restriction on $q$ being sufficiently
large in Theorem \ref{thm1.1} was removed by Casselman
\cite{Casselman}, who proved the vanishing of the middle cohomology
groups by a completely different method, using representation theory of $p$-adic groups.

Garland's vanishing theorem plays an important role in some problems 
in representation theory;  
for example, it puts strict restrictions on the continuous cohomology of 
topological groups with coefficients in infinite dimensional representations (cf. \cite{Casselman} and $\S$\ref{ss4.3}). 
An application of Garland's theorem in arithmetic geometry arises 
in the calculation of the cohomology groups of certain algebraic varieties possessing 
rigid-analytic uniformization. More precisely, 
$\fB$ can be realized as the skeleton of Drinfeld's symmetric space $\Omega^n$, so the 
cohomology groups of the algebraic variety uniformized as $\Omega^n/\G$ 
are related to the cohomology groups of $\fB/\G$ (cf. \cite{SS}).  

\subsection{Outline of the paper}
Now we give an outline of the proof of Theorem \ref{thm1.1} and the contents of this paper. 

Let $X$ 
be a finite simplicial complex of dimension $n$. Let $w$ be a ``Riemannian metric'' on $X$, by which we mean, 
following \cite{Garland2}, a function from the non-oriented simplices of $X$ to the positive real numbers.  
Let $C^i(X)$ denote the $\R$-vector space of $i$-cochains of $X$ with values in $\R$. 
Define an inner product on $C^i(X)$:  
$$
(f, g)=\sum_\sigma w(\sigma) f(\hat{\sigma}) g(\hat{\sigma}), 
$$
where the sum is over all non-oriented $i$-simplices of $X$ and $\hat{\sigma}$ is 
an oriented simplex corresponding to $\sigma$. Let $d: C^i(X)\to C^{i+1}(X)$ 
denote the coboundary operator, and $\delta: C^i(X)\to C^{i-1}(X)$ denote the 
adjoint of $d$ with respect to $(\cdot ,\cdot )$.  Let $\harm{i}(X)\subset C^i(X)$ be the subspace 
of \textit{harmonic cocycles}; by definition, these are the  
$i$-cochains annihilated by both $d$ and $\delta$.  It is not hard to show that $H^i(X, \R)\cong \harm{i}(X)$. 
This isomorphism is a consequence of  the ``Hodge decomposition'' for $C^i(X)$. In Section \ref{SecGM},  
after recalling some standard terminology related to simplicial complexes, we prove this well-known fact.  

Thus, to prove $H^i(X, \R)=0$, it suffices to prove that there are no non-zero 
harmonic cocycles. Motivated by the work of Matsushima \cite{Matsushima},  
who reduced the study of harmonic forms on real locally symmetric spaces to the 
computation of the minimal eigenvalues of certain curvature transformations, Garland reduces 
the study of $\harm{i}(X)$ to estimating the minimal non-zero eigenvalue $m^i(X)$ of 
the linear operator $\Delta=\delta d$ acting on $C^i(X)$.  Section \ref{sFI} contains a key part of 
Garland's argument: it gives the precise relationship 
between the vanishing of $\harm{i}(X)$ and lower bounds on $m^i(X)$, and it also 
gives a method for estimating $m^i(X)$ inductively. 

In Section \ref{sBTB}, we describe the Bruhat-Tits building $\fB$ of $\SL_n(K)$ as a simplicial complex 
and explain how Theorem \ref{thm1.1} follows from the results in Section \ref{sFI}, assuming 
a certain lower bound on $m^i(\Lk(v))$ for vertices of $\fB$, where $\Lk(v)$ denotes the link of the vertex $v$. 
We relegate the proof of this lower bound, which is the most technical part of the paper, to 
Section \ref{sec3}. At the end of Section \ref{sBTB} we give a brief discussion of some of the 
more recent applications of Garland's method to producing examples of groups having Kazhdan's property (T). 

Based on numerical calculations, in $\S$\ref{ss4.1} we state a conjecture about the asymptotic behaviour of the eigenvalues of 
$\Delta$ acting on $C^i(\Lk(v))$ for vertices of $\fB$, and in $\S$\ref{ss4.3} we give some evidence for this conjecture. 
None of the results of this paper are original, except possibly those in $\S$\ref{ss4.3}.


\section{Simplicial cohomology and harmonic cocycles}\label{SecGM}

In this section we recall the basic definitions from the theory of simplicial cohomology 
and prove a combinatorial analogue of the Hodge decomposition theorem. This last 
theorem identifies simplicial cohomology groups with spaces of harmonic cocycles. 
Its importance for the proof of Theorem \ref{thm1.1} is that, instead of 
proving that $H^i(\fB/\G, \R)=0$ directly, we will actually show that there are no non-zero harmonic 
$i$-cocycles on $\fB/\G$.    

\subsection{Basic concepts} An (abstract) \textit{simplicial complex} is a collection $X$ of finite nonempty
sets, called simplices, such that if $s$ is an element of $X$, so is every nonempty
subset of $s$. A nonempty subset of a simplex 
$s$ is called a \textit{face} of $s$. A \textit{simplex of dimension $i$}, or simply an \textit{$i$-simplex}, is a simplex 
with $i+1$ elements. The vertex set $\Ver(X)$ of $X$ is the union of its $0$-simplices. 
A subcollection of $X$ that is itself a complex is called a \textit{subcomplex} of $X$. 
The \textit{dimension} of $X$ is the largest dimension of one of its simplices, or is infinite if there is no such largest dimension. 

Let $s$ be a simplex of $X$. The \textit{star} of $s$ in $X$,
denoted $\St(s)$, is the subcomplex of $X$ consisting of the union
of all simplices of $X$ having $s$ as a face. The \textit{link} of
$s$, denoted $\Lk(s)$, is the subcomplex of $\St(s)$ consisting of
the simplices which are disjoint from $s$. If one thinks of $\St(s)$
as the ``unit ball'' around $s$ in $X$, then $\Lk(v)$ is the ``unit
sphere'' around $s$.

Let $X$ and $Y$ be simplicial complexes. The \textit{join} of $X$
and $Y$ is the simplicial complex $X\ast Y$ such that $s\in X\ast Y$
if either $s\in X$ or $s\in Y$, or $s=x\ast y:=\{x_0,\dots, x_i,
y_0, \dots, y_j\}$, where $x=\{x_0,\dots, x_i\}\in X$ and
$y=\{y_0,\dots, y_j\}\in Y$. It is clear that $X\ast Y$ is a
simplicial complex and $\dim(X\ast Y)=\dim(X)+\dim(Y)+1$. Note that,
as a special case of this construction, $\St(s)=s\ast \Lk(s)$.

A specific ordering of the vertices of $s$, up to an even permutation,
is called an \textit{orientation} of $s$. Each positive dimensional simplex 
has two orientations. 
Denote the set of $i$-simplices by $\widehat{S}_i(X)$, and the set of
oriented $i$-simplices by $S_i(X)$. Note that $\widehat{S}_0(X)=S_0(X)=\Ver(X)$.  
For $s\in S_i(X)$, $\bar{s}\in S_i(X)$ denotes the same simplex but with
opposite orientation. An $\R$-valued \textit{$i$-cochain} on $X$ is
a function $f: S_i(X)\to \R$ which is alternating if $i\geq 1$, i.e., $f(s)=-f(\bar{s})$. 
(A $0$-cochain is just a function on $\Ver(X)$.)
The $i$-cochains naturally form an $\R$-vector
space which is denoted $C^i(X)$. If $i<0$ or $i>\dim(X)$, we set
$C^i(X)=0$.

The \textit{coboundary} operator is the linear transformation $d:C^i(X)\to C^{i+1}(X)$ defined by
\begin{equation}\label{eq-d}
df([v_0,\dots,
v_{i+1}])=\sum_{j=0}^{i+1}(-1)^jf([v_0,\dots,\hat{v}_j,\dots,
v_{i+1}]),
\end{equation}
where $[v_0,\dots,v_{i+1}]\in S_{i+1}(X)$ and the symbol $\hat{v}_j$ means that the vertex $v_j$ is to be
deleted from the array.
The kernel of $d:C^i(X)\to C^{i+1}(X)$ is called the subspace of \textit{$i$-cocycles} and denoted $Z^i(X)$. 
As one easily verifies  $d\circ d = 0$ (cf. \cite[p. 30]{Munkres}), so $dC^{i-1}(X)\subset Z^i(X)$. 
The \textit{$i$-th cohomology group of $X$} (with real coefficients) is
$$H^i(X):=Z^i(X)/dC^{i-1}(X).$$
Let $\mathbf{1}\in C^0(X)$ be the function defined by $\mathbf{1}(v)=1$ for all $v\in \Ver(X)$. The 
subspace $\R \mathbf{1}\subset C^0(X)$ spanned by $\mathbf{1}$ is the space of constant function. It is easy to see that 
 $\R \mathbf{1} \subset Z^0(X)$. One defines the \textit{reduced $i$-th cohomology group} $\tilde{H}^i(X)$ of
$X$ by setting $\tilde{H}^i(X)=H^i(X)$ if $i\geq 1$, and $\tilde{H}^0(X)=Z^0(X)/\R \mathbf{1}$. 
It is easy to show that the ``geometric 
realization'' of $X$ is connected if and only if $\tilde{H}^0(X)=0$; see \cite[p. 256]{Munkres}.

Now assume $X$ is finite, i.e., has finitely many vertices. To each simplex $s$ of $X$ assign a positive real number $w(s)=w(\bar{s})$, 
which we call the \textit{weight} of $s$. Define an inner-product on $C^i(X)$ by
\begin{equation}\label{eq-pairing}
(f,g)=\sum_{{s} \in \widehat{S}_i(X)} w(s)f(s)g(s),\qquad f,g\in C^i(X),
\end{equation}
where in $w(s)f(s)g(s)$ we choose
some orientation of $s$; this is well-defined since $f(s)g(s)=f(\bar{s})g(\bar{s})$. 

Let $s=[v_0,\dots, v_i]\in S_i(X)$ and $v\in \Ver(X)$. If the set
$\{v,v_0,\dots, v_i\}$ is an $(i+1)$-simplex of $X$, then we denote
by $[v, s]\in S_{i+1}(X)$ the oriented simplex $[v,v_0,\dots, v_i]$.
Define a linear transformation $\delta: C^i(X)\to C^{i-1}(X)$ by
\begin{equation}\label{eq-delta}
\delta f(s)=\sum_{\substack{v\in \Ver(X)\\ [v,s]\in S_i(X)}}
\frac{w([v,s])}{w(s)}f([v,s]).
\end{equation}
This operator is the adjoint of $d$ with respect to \eqref{eq-pairing}:
\begin{lem}\label{v-prop1.12}
If $f\in C^i(X)$ and $g\in C^{i+1}(X)$, then $(df,g)=(f,\delta g)$.
\end{lem}
\begin{proof}
\begin{align*}
&(df,g)\\ 
&=\sum_{\substack{s=[v_0,\dots,v_{i+1}] \in
\widehat{S}_{i+1}(X)}}
w(s)\sum_{j=0}^{i+1}f([v_0,\dots, \hat{v}_j,\dots,
v_{i+1}])g([v_j,v_0,\dots,\hat{v}_j,\dots, v_{i+1}])\\
&=\sum_{\substack{\sigma \in \widehat{S}_{i}(X)}}
w(\sigma)f(\sigma)\sum_{\substack{v\in \Ver(X) \\ [v,\sigma]\in
S_{i+1}(X)}}\frac{w([v,\sigma])}{w(\sigma)}g([v, \sigma])=(f,\delta g).
\end{align*}
\end{proof}

The kernel of $\delta: C^i(X)\to C^{i-1}(X)$ will be denoted by $\sZ^i(X)$. 

\subsection{Combinatorial Hodge decomposition}
The intersection $$\harm{i}(X):=\sZ^i(X)\cap Z^i(X)$$ in $C^i(X)$ 
is the subspace of \textit{harmonic $i$-cocycles}. (This subspace depends on the 
choice of the inner-product \eqref{eq-pairing}.)

\begin{rem}
The term \textit{harmonic} comes from the fact that $f\in C^i(X)$ is in $\harm{i}(X)$ if and only if 
$f$ is in the kernel of the operator $d\delta+\delta d$, which is a combinatorial analogue 
of the Laplacian. 
\end{rem}

\begin{thm}\label{lem1.11} 
Relative to the inner-product \eqref{eq-pairing}, we have the orthogonal direct sum decompositions
\begin{align}
\label{eq-ci} C^i(X) &=\harm{i}(X)\oplus dC^{i-1}(X)\oplus \delta C^{i+1}(X), \\ 
\label{eq-zi} Z^i(X) &=\harm{i}(X)\oplus dC^{i-1}(X), \\  
\label{eq-hi} \sZ^i(X) &=\harm{i}(X)\oplus \delta C^{i+1}(X). 
\end{align}
This implies  
\begin{equation}\label{eq-CohomHarm}
H^i(X)=Z^i(X)/dC^{i-1}(X)\cong \sZ^i(X)/\delta C^{i+1}(X)\cong \harm{i}(X).
\end{equation}
\end{thm}
\begin{proof} Let $f\in dC^{i-1}(X)$ and $g\in \delta C^{i+1}(X)$. We can write $f=df'$ and $g=\delta g'$ for some 
$f'\in C^{i-1}(X)$ and $g'\in C^{i+1}(X)$. Since $d^2=0$, using Lemma \ref{v-prop1.12} we get 
$$
(f, g)=(df', \delta g')=(d^2 f', g')=(0, g')=0. 
$$
Hence $dC^{i-1}(X)\perp \delta C^{i+1}(X)$. Now suppose $h\in C^i(X)$ is orthogonal to $dC^{i-1}(X)\oplus \delta C^{i+1}(X)$. Then 
$$
0=(h, df')=(\delta h, f')
$$
for all $f'\in C^{i-1}(X)$, which implies $\delta h=0$. Similarly, $0=(h, \delta g')$ implies that $dh=0$. 
Therefore, $h\in \harm{i}(X)$. In other words, the orthogonal complement of $dC^{i-1}(X)\oplus \delta C^{i+1}(X)$ in $C^i(X)$ 
is $\harm{i}(X)$. This proves \eqref{eq-ci}. 

A cochain $f\in C^i(X)$ is in the orthogonal complement of $ \delta C^{i+1}(X)$ if and only if 
$(f, \delta g')=(df, g')=0$ for all $g'\in C^{i+1}(X)$, which implies $df=0$. Thus 
\begin{equation}\label{eq-ci2}
C^i(X) =Z^i(X)\oplus \delta C^{i+1}(X). 
\end{equation}
A similar argument shows that the orthogonal complement of $dC^{i-1}(X)$ is $\sZ^i(X)$: 
\begin{equation}\label{eq-ci3}
C^i(X) =\sZ^i(X)\oplus dC^{i-1}(X). 
\end{equation}
Comparing \eqref{eq-ci2} and \eqref{eq-ci3} with \eqref{eq-ci}, we get \eqref{eq-zi} and \eqref{eq-hi}. 
Finally, \eqref{eq-zi} and \eqref{eq-hi} imply \eqref{eq-CohomHarm}. 
\end{proof}

\begin{defn}\label{defn1.7} Following \cite{Garland}, we call the linear
transformation $\Delta=\delta d$ on $C^i(X)$ the \textit{curvature transformation}. 
(What we denote by $\Delta$ in this paper is denoted by $\Delta^+$ in \cite{Garland} and \cite{Borel}.)
\end{defn}


By Lemma \ref{v-prop1.12}, for any $f, g\in C^i(X)$ we have 
$$
(\Delta f, g)= (\delta d f, g)= (d f, dg) = (f, \delta d g)= (f, \Delta g)
$$
and 
\begin{equation}\label{eq-see}
(\Delta f, f)= (df, df)\geq 0. 
\end{equation}
Hence $\Delta$ is a self-adjoint positive operator on $C^i(X)$, which implies that 
$C^i(X)$ has an orthonormal basis consiting of eigenvectors of $\Delta$, and the 
eigenvalues of $\Delta$ are nonnegative. 

\begin{lem}\label{lem1.7} Let $i\geq 0$. 
\begin{enumerate}
\item The subspace of $C^i(X)$ spanned by the eigenfunctions of $\Delta$ with positive eigenvalues 
is $\delta C^{i+1}(X)=Z^i(X)^\perp$.
\item If $H^i(X)=0$, then $\delta C^{i+1}(X)=\sZ^i(X)$. 
\end{enumerate}
\end{lem}
\begin{proof} 
It is clear from \eqref{eq-see} that if $\Delta f=0$ then $df=0$. 
Hence $\ker(\Delta)\subseteq Z^i(X)$. Conversely, if $f\in Z^i(X)$ then $\Delta f=\delta df=\delta 0=0$.  
Therefore the subspace of $C^i(X)$ spanned by the eigenfunctions of $\Delta$ with positive eigenvalues 
is $Z^i(X)^\perp$. This latter subspace is $\delta C^{i+1}(X)$, as follows from \eqref{eq-ci2}. 
The second claim of the lemma follows from \eqref{eq-CohomHarm}. 
\end{proof}


\section{Garland's method}\label{sFI} 

In this section we discuss what is nowadays is called \textit{Garland's method}. 
Let $X$ be a finite simplicial complex. 
Let $$\la_{\min}^i(X):=\min_{v\in \Ver(X)} m^i(\Lk(v)).$$ 
Garland's method shows that a strong enough lower 
bound on $\la_{\min}^{i-1}(X)$ implies 
that $\harm{i}(X)=0$ (hence also $H^i(X, \R)=0$ by the Hodge decomposition 
discussed in the previous section). Moreover, the method 
gives a lower bound on $m^k(X)$ in terms of $\la_{\min}^{k-1}(X)$, 
hence allows to estimate $\la_{\min}^{i-1}(X)$ inductively in certain situations. 
(We will see an example of such an inductive estimate in Section \ref{sec3}.)
The observation that Garland's ideas from \cite{Garland} apply to any finite simplicial complex  
satisfying a certain condition is due to Borel; 
in this section we partly follow \cite[$\S$1]{Borel}. 

The proof of the main results has two parts: It starts with a decomposition of $(\Delta f, f)$  
into a sum $\sum_v (\Delta f_v, f_v)$ over the vertices of $X$, where $f_v$ 
is the restriction of $f$ to the ``unit ball'' around $v$; this is the content of $\S$\ref{ss3.1}.  
Then one bounds $(\Delta f_v, f_v)$ in terms of $m^{i-1}(\Lk(v))$ by studying the local version of the 
curvature transformation; this is the content of $\S$\ref{ss3.2}.  
We combine these two parts in $\S$\ref{ss3.3} to prove the main results.  
One of the subtleties is that to make this strategy work one has to choose 
an appropriate metric \eqref{eq-themetric}.  

In this section we assume that $X$ is a finite $n$-dimensional
complex which satisfies the following property:
\begin{center}
$(\star)$ each simplex of $X$ is a face of some $n$-simplex.
\end{center}
For $s\in S_i(X)$, let 
\begin{equation}\label{eq-themetric}
w(s) = \text{the number of (non-oriented) $n$-simplices containing $s$.}
\end{equation}
Note that, due to $(\star)$, $w(s)\neq 0$.


\subsection{Decomposition of $(\Delta f,f)$ into local factors}\label{ss3.1} We start with a simple lemma: 

\begin{lem}\label{lem-w} Let $s\in \widehat{S}_i(X)$ be fixed. Then
$$
\sum_{\substack{\sigma\in \widehat{S}_{i+1}(X)\\ s\subset
\sigma}}w(\sigma)=(n-i)w(s).
$$
\end{lem}
\begin{proof}
Given an $n$-simplex $t$ such that $s\subset t$ there are
exactly $(n-i)$ simplices $\sigma$ of dimension $(i+1)$ such that
$s\subset \sigma\subset t$. Hence in the sum of the lemma we count
every $n$-simplex containing $s$ exactly $(n-i)$ times.
\end{proof}

For a fixed $v\in \Ver(X)$ define a linear transformation $\rho_v: C^i(X)\to C^i(X)$ by:
$$
\rho_vf (s)=\left\{
  \begin{array}{ll}
    f(s) & \hbox{if } v\in s; \\
    0 & \hbox{otherwise.}
  \end{array}
\right.
$$
It is clear that 
\begin{equation}\label{eq-rho^2}
\rho_v\rho_v=\rho_v,
\end{equation}
and for $f,g\in C^i(X)$
\begin{equation}\label{eq-(rho)}
(\rho_vf, g)= (\rho_vf, \rho_v g)=(f, \rho_vg). 
\end{equation}
Moreover, since any $i$-simplex has $(i+1)$-vertices, for $f\in C^i(X)$ we
have the equality
\begin{equation}\label{eq-obv}
\sum_{v\in \Ver(X)}\rho_vf=(i+1)f.
\end{equation}

\begin{lem}\label{lem-borel} For $f\in C^i(X)$, we have
$$
i\cdot (\Delta f, f)=\sum_{v\in \Ver(X)}(\Delta \rho_v f,\rho_v f) - (n-i)(f,f).
$$
\end{lem}
\begin{proof} First, to simplify the notation in our calculations we introduce new notation. 
Let $\sigma\in S_{i+1}(X)$ and $s\in S_i(X)$ be a face of $\sigma$. 
The orientation on $\sigma$ induces an orientation on $s$; we define $[\sigma:s]=\pm 1$ 
depending on whether this induces orientation is the original orientation of $s$ or its opposite. 
With this definition, for $f\in C^{i}(X)$ we have 
$$
df(\sigma) = \sum_{\substack{s\in \widehat{S}_i(X)\\ s\subset \sigma}}[\sigma:s]f(s),
$$
where for each face $s$ of $\sigma$ we choose some 
orientation. (Note that $[\sigma:s]f(s)$ does not depend on the choice of the orientation of $s$.)

Let $v\in \sigma$ be a fixed vertex and $s_0\in \widehat{S}_i(X)$ be the unique face of $\sigma$ 
not containing $v$. Then 
$$
df(\sigma) = d\rho_v f(\sigma)+[\sigma:s_0]f(s_0). 
$$
Hence 
$$
df(\sigma)^2 = d\rho_v f(\sigma)^2+2[\sigma:s_0]f(s_0)\sum_{v\in s\subset \sigma}[\sigma:s]f(s)+f(s_0)^2. 
$$
Summing both sides over all vertices of $\sigma$ we get 
$$
(i+2)df(\sigma)^2=\sum_{v\in \sigma}d\rho_v f(\sigma)^2+2\sum_{\substack{s, s'\subset \sigma\\ s\neq s'}}[\sigma:s][\sigma:s']f(s)f(s') 
+\sum_{s\subset \sigma} f(s)^2
$$
$$
=\sum_{v\in \sigma}d\rho_v f(\sigma)^2 + 2\sum_{\substack{s, s'\subset \sigma}}[\sigma:s][\sigma:s']f(s)f(s') -\sum_{s\subset \sigma} f(s)^2
$$
$$
=\sum_{v\in \sigma}d\rho_v f(\sigma)^2 + 2df(\sigma)^2 -\sum_{s\subset \sigma} f(s)^2.
$$
Hence
\begin{equation}\label{Borel-noproof}
i\cdot df(\sigma)^2 = \sum_{v\in \sigma}d\rho_v f(\sigma)^2 -\sum_{s\subset \sigma} f(s)^2.
\end{equation}
Now
\begin{align*}
i\cdot (\Delta f, f) &\overset{\mathrm{Lem.} \ref{v-prop1.12}}{=}  i\cdot (df, df) = i\sum_{\sigma\in \widehat{S}_{i+1}(X)} w(\sigma) df(\sigma)^2
\\
&\overset{\eqref{Borel-noproof}}{=}\sum_{\sigma\in \widehat{S}_{i+1}(X)} \sum_{v\in \sigma}w(\sigma) d\rho_v f(\sigma)^2-
\sum_{\sigma\in \widehat{S}_{i+1}(X)}\sum_{s\subset \sigma} w(\sigma)f(s)^2  
\\
&=\sum_{v\in \Ver(X)} (d\rho_v f,d\rho_v f) -
\sum_{s\in \widehat{S}_{i}(X)}f(s)^2 \sum_{\substack{\sigma\in \widehat{S}_{i+1}(X)\\ s\subset \sigma}} w(\sigma) 
\\
&\overset{\mathrm{Lem.} \ref{lem-w}}{=} \sum_{v\in \Ver(X)} (\Delta\rho_v f,\rho_v f)-(n-i)\sum_{s\in \widehat{S}_{i}(X)}w(s)f(s)^2 
\\
&= \sum_{v\in \Ver(X)} (\Delta\rho_v f,\rho_v f)-(n-i)(f,f). 
\end{align*}
\end{proof}


\subsection{Local curvature transformations}\label{ss3.2}

For
$f,g\in C^i(\Lk(v))$ define their inner-product by
\begin{equation}\label{eq-locp}
(f,g)_v=\sum_{s\in \widehat{S}_i(\Lk(v))}w_v(s)\cdot f(s)\cdot
g(s),
\end{equation}
where $w_v(s)$ is the number of $(n-1)$-simplices in $\Lk(v)$
containing $s$. Note that $\Lk(v)$ is an
$(n-1)$-dimensional complex satisfying $(\star)$. 
Another simple observation is that for a simplex $\sigma$ in $\Lk(v)$ there is a one-to-one correspondence between
the $n$-simplices of $X$ containing $[v,\sigma]$ and the
$(n-1)$-simplices of $\Lk(v)$ containing $\sigma$. Hence
\begin{equation}\label{eq-w_vw}
w_v(\sigma)=w([v,\sigma]) \quad \text{for any }\sigma\in \Lk(v).  
\end{equation}

Let $$d_v: C^i(\Lk(v))\to C^{i+1}(\Lk(v))$$ be the 
coboundary operator acting on the cochains of the finite simplicial complex $\Lk(v)$. 
Let $\delta_v$ be the adjoint of $d_v$ with respect to \eqref{eq-locp}, 
and let $$\Delta_v:=\delta_vd_v.$$ 

Lemma \ref{lem-borel} essentially decomposes $\Delta$ into a sum of its restrictions $\Delta\rho_v$ to $\St(v)$ over all vertices.
We want to relate $\Delta\rho_v$ to $\Delta_v$, and hence to relate the eigenvalues of $\Delta$ 
to the eigenvalues of its local version $\Delta_v$.  
For this we need to introduce one more linear operator: For $i\geq 1$, define 
\begin{align*}
&\tau_v:C^{i}(X)\to C^{i-1}(X),\\ 
& \tau_vf(s)=\left\{
  \begin{array}{ll}
    f([v,s]), & \hbox{if $s\in S_{i-1}(\Lk(v))$;}\\
    0, & \hbox{otherwise.}
  \end{array}
\right.
\end{align*}

Given $f\in C^i(X)$, its restriction to $\Lk(v)$ defines a function in $C^i(\Lk(v))$, which 
by slight abuse of notation we denote by the same letter. With this convention, for  $f\in C^i(X)$ 
we can consider $d_v f$ and $\delta_v f$, which are now functions on $\Lk(v)$. Similarly, we 
can compute the pairing \eqref{eq-locp} on $C^i(X)$. 

\begin{lem}\label{prop7.12} Assume $i\geq 1$. 
For $f, g\in C^i(X)$, we have
$$(\tau_vf,\tau_vg)_v=(\rho_vf,\rho_vg).$$
\end{lem}
\begin{proof} We have
\begin{align*}
(\tau_vf,\tau_vg)_v &=\sum_{\sigma\in
\widehat{S}_{i-1}(\Lk(v))}w_v(\sigma)\cdot \tau_vf(\sigma)\cdot
\tau_vg(\sigma)\\ 
&\overset{\eqref{eq-w_vw}}{=}
\sum_{s\in \widehat{S}_{i}(\St(v))}w(s)\cdot \rho_vf(s)\cdot
\rho_vg(s).
\end{align*}
Since $\rho_v f$ is zero away from $\St(v)$, the last sum can be extended
to the whole $\widehat{S}_{i}(X)$, so the lemma follows.
\end{proof}

\begin{lem} Let $f\in C^i(X)$. We have 
\begin{align}
\label{eq-taud} \tau_v d\rho_v f &=-d_v\tau_v f,  \quad i\geq 1,\\
\label{eq-taudelta} \tau_v \delta f &=-\delta_v\tau_v f,  \quad i\geq 2, \\
\label{eq-tauDelta} \tau_v \Delta \rho_v f &=\Delta_v\tau_v f,  \quad i\geq 1.
\end{align}
\end{lem}
\begin{proof}
Let $s\in S_i(\Lk(v))$, $i\geq 1$. We have 
$$
\tau_v d\rho_v f(s)=d\rho_vf([v,s])=\rho_v(f(s)-f([v, ds]))=-f([v, ds]). 
$$
In the last term $ds$ denotes the image of $s$ under the boundary
operator and $[\cdot]$ is extended linearly to $\Z[S_i(\Lk(v))]$. Since $d$ restricted to $\Lk(v)$ coincides with $d_v$, 
we have $f([v, ds])=d_v\tau_v f(s)$. 
This proves \eqref{eq-taud}. 

Now assume $i\geq 2$ and let $s\in S_{i-2}(\Lk(v))$. We have 
$$
\tau_v \delta f(s)=\delta f([v, s])= \sum_{\substack{x\in \Ver(X)\\ [x,v, s]\in S_{i}(X)}}\frac{w([x,v,s])}{w([v, s])}f([x,v,s])
$$
$$
\overset{\eqref{eq-w_vw}}{=}- \sum_{\substack{x\in \Ver(\Lk(v))\\ [x, s]\in S_{i-1}(\Lk(v))}}\frac{w_v([x, s])}{w_v(s)}\tau_v f([x,s])=-\delta_v\tau_v f(s). 
$$
This proves \eqref{eq-taudelta}. 

Finally, 
$$
\tau_v \Delta \rho_v f =\tau_v \delta d \rho_v f\overset{\eqref{eq-taudelta}}{=}-\delta_v \tau_v d\rho_v f
\overset{\eqref{eq-taud}}{=} \delta_v d_v\tau_v f=\Delta_v\tau_v f. 
$$
This proves \eqref{eq-tauDelta}. 
\end{proof}

\begin{lem}\label{prop7.14} Assume $i\geq 1$. For $f\in C^i(X)$, we have
$$(\Delta \rho_v f,\rho_v f)=(\Delta_v\tau_vf,\tau_vf)_v.$$
\end{lem}
\begin{proof} We have 
$$
(\Delta \rho_v f,\rho_vf)\overset{\eqref{eq-(rho)}}{=}(\rho_v\Delta \rho_v
f,\rho_vf)\overset{\mathrm{Lem.} \ref{prop7.12}}{=}(\tau_v\Delta \rho_vf,\tau_vf)_v 
\overset{\eqref{eq-tauDelta}}{=}(\Delta_v\tau_vf,\tau_vf)_v.
$$
\end{proof}

\begin{notn} Given a simplicial complex $Y$, let
$M^i(Y)$ and $m^i(Y)$ be the maximal and minimal non-zero
eigenvalues of $\Delta$ acting on $C^i(Y)$, respectively. Denote
$$
\la^i_{\max}(Y)= \max_{\substack{v\in \Ver(Y)}}M^i(\Lk(v)), 
$$
$$
 \la^i_{\min}(Y)= \min_{\substack{v\in
\Ver(Y)}}m^i(\Lk(v)).
$$
\end{notn}

\begin{lem}\label{lemDec5} Assume $i\geq 1$. For $f\in C^i(X)$, we have
$$
(\Delta_v \tau_v f, \tau_v f)_v\leq M^{i-1}(\Lk(v))\cdot(\tau_v f,
\tau_v f)_v.
$$
If $\tilde{H}^{i-1}(\Lk(v))=0$, then for $f\in \sZ^i(X)$ we have
$$
(\Delta_v \tau_v f, \tau_v f)_v\geq m^{i-1}(\Lk(v))\cdot(\tau_v f,
\tau_v f)_v.
$$
\end{lem}
\begin{proof}
We can choose an orthonormal basis $\{e_1, \dots, e_h\}$ of
$C^{i-1}(\Lk(v))$ with respect to $(\cdot , \cdot)_v$ which consists
of $\Delta_v$-eigenvectors. Let $\{\kappa_1,\dots, \kappa_h\}$ be
the corresponding eigenvalues. We have $\kappa_j\geq 0$ for all $j$. Write
$\tau_vf=\sum_{j=1}^h a_j e_j$. Then
$$
(\Delta_v \tau_v f, \tau_v f)_v  = (\sum_{j=1}^h a_j \kappa_j e_j,
\sum_{j=1}^h a_j e_j)_v=\sum_{j=1}^h a_j^2 \kappa_j 
$$
$$
\leq M^{i-1}(\Lk(v)) \sum_{j=1}^h a_j^2 =
M^{i-1}(\Lk(v)) (\tau_v f, \tau_v f)_v.
$$

The second claim will follow from a similar argument if we show
that $\tau_vf$ belongs to the subspace of $C^{i-1}(\Lk(v))$ spanned by $\Delta_v$-eigenfunctions
with \textbf{positive} eigenvalues. 
First assume $i\geq 2$. In this case ${H}^{i-1}(\Lk(v))=\tilde{H}^{i-1}(\Lk(v))=0$.  
Hence, thanks to Lemma \ref{lem1.7}, it is enough to show that $\tau_vf \in\cH^{i-1}(\Lk(v))$. 
Since by assumption $\delta f=0$, from \eqref{eq-taudelta} we get 
$\delta_v\tau_v f=-\tau_v \delta f=0$. Thus $\tau_v f\in \sZ^{i-1}(\Lk(v))$. 

Now assume $i=1$ and $\tilde{H}^0(\Lk(v))=0$. This last assumption is equivalent to $\Lk(v)$ 
being connected. In this case $Z^0(\Lk(v))$ is spanned by the function 
$\mathbf{1}\in C^0(\Lk(v))$ which assumes value $1$ on all vertices of $\Lk(v)$.  
By Lemma \ref{lem1.7}, we need to show that $\mathbf{1}$ is orthogonal to 
$\tau_v f$ with respect to the inner-product \eqref{eq-locp}. We compute
\begin{align}
\label{eq-1tuaf} (\mathbf{1}, \tau_v f)_v &=\sum_{x\in
\Ver(\Lk(v))}{w}_v(x)\cdot \tau_vf(x) \overset{\eqref{eq-w_vw}}{=}\sum_{x\in
\Ver(\Lk(v))}{w}([v,x]) f([v,x])\\ 
\nonumber &=-w(v)\delta f(v)=0. 
\end{align}
\end{proof}

\subsection{Fundamental inequalities}\label{ss3.3} Now we are ready to prove the main results of this section. 

\begin{thm}\label{thmFI} Assume $i\geq 1$. For $f\in C^i(X)$ we have 
$$
i\cdot (\Delta f, f)\leq \left((i+1)\cdot
\la^{i-1}_{\max}(X)-(n-i)\right)(f,f).
$$
If $\tilde{H}^{i-1}(\Lk(v))=0$ for every $v\in \Ver(X)$, then for $f\in \sZ^i(X)$ we have 
$$
i\cdot (\Delta f, f)\geq \left((i+1)\cdot
\la^{i-1}_{\min}(X)-(n-i)\right)(f,f).
$$
\end{thm}
\begin{proof} Combining Lemma \ref{lem-borel} with Lemma \ref{prop7.14}, we get 
\begin{equation}\label{eq-locexp}
i\cdot (\Delta f, f)=\sum_{v\in \Ver(X)}(\Delta_v \tau_v f,\tau_v f)_v - (n-i)(f,f).
\end{equation}
If $\tilde{H}^{i-1}(\Lk(v))=0$ for every $v\in \Ver(X)$ and $f\in \sZ^i(X)$, then 
by Lemma \ref{lemDec5}
\begin{align*}
\sum_{v\in \Ver(X)}(\Delta_v \tau_v f,\tau_v f)_v & \geq \sum_{v\in \Ver(X)}m^{i-1}(\Lk(v))(\tau_v f,\tau_v f)_v \\
& \geq \la^{i-1}_{\min}(X) \sum_{v\in \Ver(X)}(\tau_v f,\tau_v f)_v. 
\end{align*}
On the other hand, 
$$
\sum_{v\in \Ver(X)}(\tau_v f,\tau_v f)_v \overset{\mathrm{Lem.} \ref{prop7.12}}{=} \sum_{v\in \Ver(X)}(\rho_v f,\rho_v f) 
\overset{\eqref{eq-(rho)}}{=} \sum_{v\in \Ver(X)}(\rho_v f, f) 
$$
$$
= (\sum_{v\in \Ver(X)}\rho_v f, f)  \overset{\eqref{eq-obv}}{=} (i+1)(f,f). 
$$
Hence 
\begin{equation}\label{eq-suminequl}
\sum_{v\in \Ver(X)}(\Delta_v \tau_v f,\tau_v f)_v\geq \la^{i-1}_{\min}(X)   (i+1)(f,f).
\end{equation}
Substituting this inequality into \eqref{eq-locexp} we get the second claim of the theorem. The first claim 
follows from a similar argument. 
\end{proof} 

\begin{cor}\label{CorFI} Assume $i\geq 1$. If $\tilde{H}^{i-1}(\Lk(v))=0$ for every $v\in \Ver(X)$ and
$\la^{i-1}_{\min}(X)>\frac{n-i}{i+1}$, then $\harm{i}(X)=0$.
\end{cor}
\begin{proof} 
Let $f\in \harm{i}(X)$. Then $df=\delta f=0$. 
We obviously have $(\Delta f, f)=(df, df)=0$. 
Under our current assumptions, Theorem \ref{thmFI} then implies $(f,f)\leq 0$,
which implies $f=0$.
\end{proof}


\begin{thm}\label{thmFIeigen} Assume $i\geq 1$. 
We have 
$$ 
i\cdot M^i(X)\leq (i+1)\cdot \la^{i-1}_{\max}(X)-(n-i),
$$
and 
$$ 
i\cdot m^i(X)\geq (i+1)\cdot \la^{i-1}_{\min}(X)-(n-i).
$$
\end{thm}
\begin{proof}
Let $f\in C^i(X)$ be an eigenfunction of $\Delta$ with non-zero eigenvalue $c\neq 0$. 
Then, by Lemma \ref{lem1.7} (1), $f=\delta g$ for some $g\in C^{i+1}(X)$. 
By \eqref{eq-taudelta} $$\tau_v f=\tau_v \delta g=-\delta_v\tau_v g.$$ 
Hence, again by Lemma \ref{lem1.7} (1), $\tau_v f$ belongs to the subspace of $C^{i-1}(\Lk(v))$ spanned by 
the eigenfunctions of $\Delta_v$ with positive eigenvalues. As in the proof of Lemma \ref{lemDec5}, 
this implies $$(\Delta_v \tau_v f, \tau_v f)_v\geq m^{i-1}(\Lk(v)) (\tau_v f, \tau_v f)_v \geq \la^{i-1}_{\min}(X) (\tau_v f, \tau_v f)_v.$$ 
This inequality, as in the proof of Theorem \ref{thmFI},  implies \eqref{eq-suminequl}. 
Combining \eqref{eq-suminequl} with \eqref{eq-locexp} we get 
$$
ic(f, f)=i(\Delta f, f)\geq \left((i+1)\cdot \la^{i-1}_{\min}(X)-(n-i)\right)(f,f),  
$$
which implies the second inequality of the theorem. The first inequality can be proven by a similar argument.  
\end{proof}

\begin{rem}
Note that in Theorem \ref{thmFIeigen} we do not assume $\tilde{H}^{i-1}(\Lk(v))=0$. 
Also, since $\Delta$ is not the zero operator, we must have $M^i(X)>0$ for $i\leq n-1$, which 
implies $\la_{\max}^{i-1}(X)>\frac{n-i}{i+1}$. 
\end{rem}

\begin{notn} For $m\geq 1$, let $I_m$ denote the $m\times m$ identity matrix and let $J_m$
denote the $m\times m$ matrix whose entries are all equal to $1$.
The minimal polynomial of $J_m$ is $x(x-m)$.
\end{notn}

\begin{example}\label{Ex-P} Let $X$ be an $n$-simplex. We claim that the
eigenvalues of $\Delta$ acting on $C^i(X)$ are $0$ and $(n+1)$ for
any $0\leq i \leq n-1$. It is easy to see that $0$ is an eigenvalue,
so we need to show that the only non-zero eigenvalue of $\Delta$ is
$(n+1)$. It is enough to show
that $m^i(X)=M^i(X)=n+1$. First, suppose $i=0$. Since for any
simplex of $X$ there is a unique $n$-simplex containing it, one
easily checks that $\Delta$ acts on $C^0(X)$ as the matrix
$(n+1)I_{n+1}-J_{n+1}$. The only eigenvalues of this matrix are $0$
and $(n+1)$. Now let $i\geq 1$.
 The link of any vertex is an $(n-1)$-simplex, so by induction
$\la_{\min}^{i-1}(X)=\la^{i-1}_{\max}(X)=n$. 
Theorem \ref{thmFIeigen} implies
$$
i\cdot M^i(X)\leq (i+1)n-(n-i)=i(n+1)
$$
and
$$
i\cdot m^i(X)\geq (i+1)n-(n-i)=i(n+1).
$$
Hence $(n+1)\leq m^i(X)\leq M^i(X)\leq (n+1)$, which implies the
claim. (Of course, for this simple example it is possible, but not completely trivial, to compute
the eigenvalues of $\Delta$ directly.)
\end{example}

\begin{notn} Let $0\leq j\leq n-1$. Given $s\in \widehat{S}_j(X)$, 
its link $\Lk(s)$ in $X$ has dimension $n-(j+1)$ and satisfies $(\star)$. For $0\leq i\leq n-(j+1)$, denote 
$$
\la^{i,j}_{\max}(X)= \max_{\substack{s\in \widehat{S}_j(X)}}M^i(\Lk(s)), 
$$
$$
 \la^{i,j}_{\min}(X)= \min_{\substack{s\in \widehat{S}_j(X)}}m^i(\Lk(s)).
$$
With our earlier notation, we have $\la^{i,0}_{\max}(X)=\la^{i}_{\max}(X)$ and $\la^{i,0}_{\min}(X)=\la^{i}_{\min}(X)$. 
\end{notn}

\begin{cor}\label{corFIeigen} Let $0\leq j<i\leq n-1$. 
We have 
$$ 
(i-j)\cdot M^i(X)\leq (i+1)\cdot \la^{i-(j+1),j}_{\max}(X)-(j+1)(n-i).
$$
and 
$$ 
(i-j)\cdot m^i(X)\geq (i+1)\cdot \la^{i-(j+1),j}_{\min}(X)-(j+1)(n-i).
$$
\end{cor}
\begin{proof} We will prove the second inequality. The first inequality can be proven 
by a similar argument. 

If $j=0$, then the claim is just Theorem \ref{thmFIeigen}. Now, given $j\geq 1$, assume 
that we proved the inequality for $j-1$: 
\begin{equation}\label{eqNN1}
(i-(j-1))\cdot m^i(X)\geq (i+1)\cdot \la^{i-j,j-1}_{\min}(X)-j(n-i).
\end{equation}
Let $s\in \widehat{S}_{j-1}(X)$. The dimension of $\Lk(s)$ is $n-j$, so by Theorem \ref{thmFIeigen} we have 
$$
(i-j)m^{i-j}(\Lk(s))\geq (i-j+1)\la^{i-j-1,0}_{\min}(\Lk(s))-((n-j)-(i-j)). 
$$
On the other hand, the link of a vertex $v\in \Lk(s)$ in $\Lk(s)$ is the same as the 
link of the $j$-simplex $[v, s]$ in $X$. Thus, 
$$
\la^{i-j-1,0}_{\min}(\Lk(s))\geq \la^{i-j-1,j}_{\min}(X),
$$
and 
$$
(i-j)m^{i-j}(\Lk(s))\geq (i-j+1)\la^{i-j-1,j}_{\min}(X)-(n-i). 
$$
Taking the minimum over all $s\in \widehat{S}_{j-1}(X)$, we get 
\begin{equation}\label{eqNN2}
(i-j)\la^{i-j,j-1}_{\min}(X)\geq (i-j+1)\la^{i-j-1,j}_{\min}(X)-(n-i). 
\end{equation}
Substituting \eqref{eqNN2} into \eqref{eqNN1}, gives the desired inequality for $j$. 
\end{proof}

The argument in the previous proof can be easily adapted to prove the following inequality:
$$
(i+1)\la^{i-1,0}_{\min}(X)-(n-i)\geq \frac{i}{i-j}\left((i+1)\la^{i-j-1,j}_{\min}(X)-(j+1)(n-i)\right). 
$$
Now substituting this inequality into the second inequality of Theorem \ref{thmFI}, we get: 
\begin{cor}\label{corNN1}
Assume $0\leq j< i\leq n-1$. 
If $\tilde{H}^{i-1}(\Lk(v))=0$ for every $v\in \Ver(X)$, then for $f\in \sZ^i(X)$ we have 
$$
(i-j)\cdot (\Delta f, f)\geq \left((i+1)\cdot \la^{i-j-1,j}_{\min}(X)-(j+1)(n-i)\right)(f,f).
$$
In particular, if $\la^{i-j-1,j}_{\min}(X)>\frac{(j+1)(n-i)}{(i+1)}$, then $\harm{i}(X)=0$. 
\end{cor}

\begin{rem}\label{remAddedLast} Using \eqref{eqNN2} it is easy to check 
that $\la^{i-j-1,j}_{\min}(X)>\frac{(j+1)(n-i)}{(i+1)}$ implies 
$\la^{i-k-1,k}_{\min}(X)>\frac{(k+1)(n-i)}{(i+1)}$ for all $0\leq k\leq j$. 
Hence the strongest assumption for Corollary \ref{corNN1} is $\la^{0,i-1}_{\min}(X)>\frac{i(n-i)}{(i+1)}$. 
The advantage in trying to prove this last inequality, besides the fact that it implies all the others, 
is that the question about the vanishing of $H^i(X)$ reduces to estimating the minimal non-zero eigenvalues of laplacians on graphs. 
\end{rem}

For the purposes of proving Theorem \ref{thm1.1} we will need a variant of Corollary \ref{CorFI}. 
Let $X$ be an $n$-dimensional simplicial complex satisfying $(\star)$ but which is
not necessarily finite. Let $\G$ be a group \textit{acting} on $X$.
This means that $\G$ acts on the vertices of $X$ and preserves the
simplicial structure of $X$, i.e., whenever the vertices
$\{v_0,\dots, v_i\}$ of $X$ form an $i$-simplex, $0\leq i\leq
\dim(X)$, then for any $\gamma\in \G$ the vertices $\{\gamma
v_0,\dots,\gamma v_i\}$ also form an $i$-simplex. Consider the
following condition on the action of $\G$:
\begin{center}
$(\dag)\quad \St(v)\cap \St(\gamma v)=\varnothing$ for any $v\in \Ver(X)$ and any
$1\neq \gamma\in \G$.
\end{center}
In particular, this implies that the stabilizer of any simplex is trivial.

\begin{defn}
Let $X/\G$ be the simplicial complex whose vertices $\Ver(X/\G)$ are
the orbits $\Ver(X)/\G$ and a subset $\{\tilde{v}_0,\dots,
\tilde{v}_i\}$ of $\Ver(X/\G)$ forms an $i$-simplex, $0\leq i\leq
\dim(X)$, if we can choose a representative $v_j\in \Ver(X)$ from
the orbit $\tilde{v}_j$ for each $0\leq j\leq i$ so that that
$\{v_0, \dots, v_i\}$ form an $i$-simplex in $X$.
\end{defn}

It is obvious that $X/\G$ is a simplicial complex. Moreover, if
$(\dag)$ holds then $S_i(X/\G)$ is in bijection with the orbits
$S_i(X)/\G$ for any $i$. Indeed, as is easy to check, $(\dag)$
implies that if $\{v_0,\dots, v_{i}\}$ and $\{\gamma_0 v_0,\dots,
\gamma_{i} v_{i}\}$ are in $\widehat{S}_i(X)$ for some
$\gamma_0,\dots, \gamma_i\in \G$, then $\gamma_0=\cdots=\gamma_i$.

\begin{cor}\label{CorFII} Let $\G$ be a group acting on
$X$ so that $(\dag)$ is satisfied. Assume $X/\G$ is finite and $i\geq 1$. If $\tilde{H}^{i-1}(\Lk(v))=0$ for every $v\in \Ver(X)$ and
$\la^{i-1}_{\min}(X)>\frac{n-i}{i+1}$, then $H^i(X/\G)=0$.
\end{cor}
\begin{proof} $X/\G$ is a finite $n$-dimensional complex
satisfying $(\star)$. Let $v\in \Ver(X)$ and let $\tilde{v}$ be the
image of $v$ in $X/\G$. Due to $(\dag)$, $\Lk(v)\cong
\Lk(\tilde{v})$. Hence 
$\tilde{H}^{i-1}(\Lk(\tilde{v}))=0$ for every $\tilde{v}\in
\Ver(X/\G)$ and $\la^{i-1}_{\min}(X/\G)>\frac{n-i}{i+1}$. By Corollary \ref{CorFI}, we have 
$\harm{i}(X/\G)=0$. Therefore, by Theorem \ref{lem1.11}, $H^i(X/\G)=0$. 
\end{proof}


\section{The Bruhat-Tits building of $\SL_n(K)$}\label{sBTB}

In this section we describe the Bruhat-Tits building of $\SL_n$, and the links of its vertices. Then, assuming a certain lower  
bound on the minimal non-zero eigenvalue of the curvature transformation 
acting on the links, we prove Theorem \ref{thm1.1}. 
This is an important application of the ideas developed in the previous section. 
(The required lower bound on the minimal non-zero eigenvalue will be proven in Section \ref{sec3}.)
We conclude the section with a brief discussion of some applications 
of Garland's method to produce examples of groups having the so-called property (T). 

\subsection{The building} Let $n\geq 0$ be a non-negative integer. 
Let $K$ be a complete discrete valuation field. Let $\cO$ be the
ring of integers in $K$, $\pi$ be a uniformizer, $\cO/\pi \cO\cong \F_q$, 
where $\F_q$ denotes the finite field with $q$ elements. 
be the residue field, and $q$ be the order of $k$.
Let $\cV$ be an $(n+2)$-dimensional vector space over $K$. A
subset $L\subset \cV$ which has a structure of a free $\cO$-module
of rank $(n+2)$ such that $L\otimes_\cO K=\cV$ is called a
\textit{lattice}. It is clear that if $L$ is a lattice then
$xL:=\{x\cdot \ell\ |\ \ell\in L\}$, $x\in K^\times$, is also a
lattice. We say that $L$ and $xL$ are \textit{similar}. Similarity
defines an equivalence relation on the set of lattices in $\cV$. We
denote the equivalence class of $L$ by $[L]$.

The \textit{Bruhat-Tits building of $\SL_{n+2}(K)$} is the
simplicial complex $\fB_n$ with set of vertices $\{[L]\ |\ L \text{ is
a lattice in }\cV\}$ where $\{[L_0],\dots,[L_i]\}$ form an $i$-simplex 
if  there is $L'_j\in [L_j]$ for each $j$ with
$$
\pi L_i'\subsetneqq L_0'\subsetneqq L_1'\subsetneqq\cdots\subsetneqq L'_i.
$$

To visualize $\fB_n$ in some way, fix a basis
$\{e_1,\dots, e_{n+2}\}$ of $\cV$. It is easy to see that the classes of 
lattices 
$$
\cO\pi^{a_1} e_1\oplus\cdots\oplus \cO\pi^{a_{n+2}} e_{n+2}, \qquad a_1,\dots, a_{n+2}\in \Z, 
$$
are in bijection with the elements of $\Z^{n+2}/\Z\cdot (1,1,\dots, 1)$. In particular, $\fB_n$ 
is infinite. Next, the vertices corresponding to $(a_1, \dots, a_{n+2})$ and $(b_1, \dots, b_{n+2})$ 
are adjacent in $\fB_n$ if and only if modulo $\Z\cdot (1,1,\dots, 1)$ we have $a_i\leq b_i\leq a_i+1$ for all $i$. 
For example, when $n=0$, these vertices form an infinite line as in Figure \ref{Fig1}. 
When $n=1$, these vertices give a triangulation of $\R^2$ part of which looks like Figure \ref{Fig2}. 
It is important to stress that the vertices that we considered above do not give all vertices of $\fB_n$, 
but only of a part of the building, called an \textit{apartment}. For example, it is not hard to see 
that $\fB_0$ is an infinite tree in which every vertex is adjacent to exactly $(q+1)$ other vertices. 
Similarly, $\fB_n$ is very symmetric 
in the sense that  the simplicial complexes $\St(v)$, $v\in \Ver(\fB_n)$, are all isomorphic to each other. 
To see what this complex is take a lattice $\cL$ corresponding to $v$. Then $\cL/\pi\cL=
\F_q^{n+2}=:V$, and the vertices of $\St(v)$ are in one-to-one
correspondence with the positive dimensional linear subspaces of $V$
($v$ itself corresponds to $V$). Let $\{V_0,\dots,V_i\}$ be the
linear subspaces corresponding to the vertices $\{v_0,\dots,v_i\}$
of $\St(v)$. Then $\{v_0,\dots,v_i\}$ form an $i$-simplex if and
only if the linear subspaces $\{V_0,\dots,V_i\}$ fit into an ascending sequence 
(possibly after reindexing):
$$
\cF: V_0\subset V_1\subset \cdots \subset V_i.
$$
The $i$-simplices of $\Lk(v)$ correspond to those $\cF$ for which 
$V_i\neq V$. Next, we consider more carefully $\Lk(v)$ as a simplicial complex. 

\begin{figure}
\begin{tikzpicture}[scale=1.5, semithick, inner sep=.5mm, vertex/.style={circle, fill=black}]

\node (0) at (0, 0) {(0,0)};
\node (1) at (1, 0) {(1,0)};
\node (2) at (2, 0) {(2,0)};
\node (3) at (3, 0) {};
\node (-1) at (-1, 0) {(0,1)};
\node (-2) at (-2, 0) {(0,2)};
\node (-3) at (-3, 0) {};

\path[]
(0) edge (1)
(1) edge (2)
(2) edge[dashed] (3)
(0) edge (-1)
(-1) edge (-2)
(-2) edge[dashed] (-3);
 
\end{tikzpicture}
\caption{}\label{Fig1}
\end{figure}

\begin{figure}
\begin{tikzpicture}[scale=1, semithick, inner sep=.5mm, vertex/.style={circle, fill=black}]

\node (00) at (0, 0) {(0,0,0)};
\node (20) at (2, 0) {(1,0,0)};
\node (11) at (1, 1) {(1,1,0)};
\node (-11) at (-1, 1) {(0,1,0)};
\node (-20) at (-2, 0) {(0,1,1)};
\node (-1-1) at (-1, -1) {(0,0,1)};
\node (1-1) at (1, -1) {(1,0,1)};

\path[]
(00) edge (20) (00) edge (11) (00) edge (-11) (00) edge (-20) (00) edge (-1-1) (00) edge (1-1);
\path[]
(20) edge (11) (11) edge (-11) (-11) edge (-20) (-20) edge (-1-1) (-1-1) edge (1-1) (1-1) edge (20);
 
\end{tikzpicture}
\caption{}\label{Fig2}
\end{figure}

\subsection{Complexes of flags}\label{ssCF}
Fix some $n\geq 0$. Let $V$ be a linear space of
dimension $n+2$ over the finite field $\F_q$. A \textit{flag} in $V$ is an ascending sequence
\begin{equation}\label{eq-flag}
\cF: F_0\subset F_1\subset \cdots \subset F_i
\end{equation}
of distinct linear subspaces $F_0,\dots, F_i$ of $V$ such that
$F_0\neq 0$ and $F_i\neq V$. The \textit{length} of $\cF$ is
$i$. We will refer to a flag of length $i$ as $i$-flag.
In particular, the $0$-flags are simply the proper non-zero linear
subspaces of $V$. Given two flags
$$
\cF: F_0\subset F_1\subset \cdots \subset F_i\quad \text{and} \quad
\cG: G_0 \subset G_1\subset \cdots \subset G_j
$$
we say that $\cG$ \textit{refines} $\cF$, and write $\cF\prec \cG$, if for
every $0\leq k\leq i$ there is $0\leq t\leq j$ such that $F_k=G_t$.
It is convenient also to have the empty flag $\varnothing$, which is
the empty sequence of linear subspaces; we put $-1$ for the length of 
$\varnothing$. The refinement defines a partial ordering
on the set of flags in $V$; the empty flag is refined by every
other flag.

Let $\cF$ be a fixed flag of length $\ell$. Consider the following simplicial
complex $X_\cF$ (when $\cF=\varnothing$ we will also denote this
complex by $X_\varnothing^n$). The vertices of $X_\cF$ are the
$(\ell+1)$-flags refining $\cF$. The vertices $v_0,\dots, v_h$ form
an $h$-simplex if the corresponding flags are all refined by a
single $(\ell+h+1)$-flag. It is easy to see that $X_\cF$ is indeed a
finite simplicial complex of dimension $n-1-\ell$. Since any flag can be refined into an $n$-flag, $X_\cF$ satisfies $(\star)$. 
Note that the link of a vertex of the Bruhat-Tits building $\fB_n$ is isomorphic to $X_\varnothing^n$. 

Now assume $\cF\neq \varnothing$. Let $\cF: F_0\subset F_1\subset
\cdots \subset F_{\ell}$, with $\ell\geq 0$.  Consider
the array of integers $(t_0, t_1,\dots, t_{\ell+1})$ defined by
\begin{equation}\label{eq-t}
t_j=\left\{
  \begin{array}{ll}
    \dim(F_0), &  \mbox{if }j=0; \\
    \dim(F_j)-\dim(F_{j-1}), &  \mbox{if }1\leq j\leq \ell; \\
    \dim(V)-\dim(F_\ell), &  \mbox{if }j=\ell+1.
  \end{array}
\right.
\end{equation}
It is not hard to see that 
\begin{equation}\label{eqXFdecomp}
X_\cF\cong X_\varnothing^{t_0-2}\ast
X_\varnothing^{t_1-2}\ast\cdots \ast X_\varnothing^{t_{\ell+1}-2},
\end{equation}
where $X_\varnothing^{-1}$ denotes the empty complex. 

\begin{lem}\label{lemn2.1}
If $\dim(X_\cF)>0$ then $X_\cF$ is connected.
\end{lem}
\begin{proof} First we show that $X_\varnothing^n$ is connected if $n\geq 1$. Let
$x$ and $y$ be two vertices of $X_\varnothing^n$, and let $W_1$ and
$W_2$ be the corresponding subspaces of $V$. Choose a
$1$-dimensional subspace $L_i$ of $W_i$, $i=1,2$. Consider the
subspace $P:=L_1+L_2$ of $V$. Since $n\geq 1$, $P\neq V$, so
$P$ gives a vertex of $X_\varnothing^n$, which we denote by the same
letter. The vertex $P$ is adjacent to both $L_1$ and $L_2$, $L_1$
is adjacent to $x$, and $L_2$ is adjacent to $y$, so there is a
path from $x$ to $y$.

Now assume $\cF\neq \varnothing$ and $X_\cF$ is given by \eqref{eqXFdecomp}. If
$\dim(X_\cF)\geq 1$, then either at least two of
$X_\varnothing^{t_j-2}$'s are non-empty or at least one
$X_\varnothing^{t_j-2}$ has dimension $1$. In either case $X_\cF$ is
clearly connected.
\end{proof}

\begin{thm}\label{thm-G} Assume $N:=\dim(X_\cF)\geq 1$. Then for any $0\leq i\leq N-1$  
and $\eps>0$ there is a constant $q(\eps, n)$ depending only on
$\eps$ and $n$ such that $m^i(X_\cF)\geq N-i-\eps$ once $q>q(\eps, n)$.
\end{thm}

The proof of this theorem is quite complicated and will be given in Section \ref{sec3}. 

\begin{cor}\label{cor-vanishing} There is a constant $q(n)$ depending only of $n$ 
such that if $q>q(n)$, then $\tilde{H}^i(X_\cF)=0$ for all $0\leq i\leq N-1$.  
\end{cor}
\begin{proof} 
We use induction on $N$ and $i$. 
When $N=1$ or $i=0$, the claim follows from Lemma \ref{lemn2.1}, since 
$\tilde{H}^0(X_\cF)=0$ is equivalent to $X_\cF$ being connected. Now assume $N>1$ and $i\geq 1$. 
For any $v\in \Ver(X_\cF)$, we have $\Lk(v)\cong X_\cG$ for some $\cG\succ \cF$. Since $\dim(X_\cG)=N-1$, 
by the induction assumption $\tilde{H}^{i-1}(X_\cG)=0$ for $q$ large enough. 
Next, choosing $\eps$ small enough in Theorem \ref{thm-G}, 
we can make
$$
\la^{i-1}_{\min}(X_\cF)> \frac{N-i}{i+1}. 
$$
Now the assumptions of  Corollary \ref{CorFI} are satisfied, so $\tilde{H}^i(X_\cF)=0$. 
\end{proof}

\subsection{Main theorem}
Since $\Lk(v)$ is isomorphic to the simplicial complex
$X^n_\varnothing$, the complex $\fB_n$ is $(n+1)$-dimensional and satisfies $(\star)$. 

\begin{thm}\label{thm3.1}
Let $\G$ be a group acting on $\fB_n$ so that $(\dag)$ is satisfied. Assume
$\fB_n/\G$ is finite. There is a constant $q(n)$ depending only on $n$
such that if $q>q(n)$ then $H^i(\fB_n/\G)=0$ for $1\leq i\leq n$.
\end{thm}
\begin{proof}
Since $\Lk(v)\cong X_\varnothing^n$ for any $v\in \Ver(\fB_n)$,
Theorem \ref{thm-G} and Corollary \ref{cor-vanishing} imply that there is a constant $q(n)$
depending only on $n$ such that if $q>q(n)$ then for any $1\leq
i\leq n$ we have $\tilde{H}^{i-1}(\Lk(v))=0$ and
$\la_{\min}^{i-1}(\fB_n)=m^{i-1}(X_\varnothing^n)>\frac{n+1-i}{i+1}$.
Now the claim follows from Corollary \ref{CorFII}.
\end{proof}

\begin{rem}\label{rem-ST} It is known that the
cohomology groups $\tilde{H}^i(X_\cF)$ vanish for $0\leq i\leq
N-1$, without any assumptions on $q$, by a general result
of Solomon and Tits; see Appendix II in \cite{Garland} for a proof. 
Assuming this result, to prove Theorem \ref{thm3.1} we only need 
the bound $m^i(X^n_\varnothing)\geq n-i-\eps$. On the other hand, the proof of   
Theorem \ref{thm-G} is inductive, and requires proving this bound for all $X_\cF$. 
Another observation is that to prove Theorem \ref{thm3.1} it is enough to 
prove $m^0(X_\cF)\geq N-\eps$ for all $\cF$. Indeed, the link of any simplex in $\fB_n$ 
is isomorphic to some $X_\cF$, so one can appeal to Corollary \ref{corNN1} 
to get the vanishing of the cohomology. 
\end{rem}

\begin{rem}\label{rem4.2} In $\S$\ref{ss4.1} we will compute 
that $m^0(X^1_\varnothing)>1/2$ and $m^0(X^2_\varnothing)>1$.
Hence for $n=1, 2$ Garland's method proves the vanishing of $H^1(\fB_n/\G)$ for
all $q$. On the other hand,
$m^1(X^2_\varnothing)=1/3$ when $q=2$. To apply Corollary \ref{CorFII} to show that $H^2(\fB_2/\G)=0$ we need
$\la^1_{\min}(\fB_2/\G)>1/3$, so 
we need to assume $q>2$.
\end{rem}

There is an abundance of groups $\G$ satisfying $(\dag)$. The most
important examples of such groups come from arithmetic. One possible 
construction proceeds as follows. Let $F=\F_q(T)$ be the field of rational functions 
in indeterminate $T$ with $\F_q$ coefficients. Fix a place $\infty=1/T$ of $F$. 
Let $A=\F_q[T]$ be the polynomial ring. Let $K=\F_q(\!(1/T)\!)$ be the completion
of $F$ at $\infty$. Let $D$ be a central division algebra over $F$
of dimension $(n+2)^2$. Assume $D$ is split at $\infty$, i.e.,
$D\otimes_F K\cong \mathrm{Mat}_{n+2}(K)$. Let $\cD$ be a
maximal $A$-order in $D$; see \cite{Reiner} for the definitions. Let $\cD^\times$ be the 
group of multiplicative units in $\cD$. 
The quotient $\cD^\times/\F_q^\times$ can be identified with 
a discrete, cocompact subgroup of $\PGL_{n+2}(K)$. Replacing $\cD^\times/\F_q^\times$ 
by a subgroup $\Gamma\subset \cD^\times/\F_q^\times$ of finite index if necessary, 
we get a group which naturally acts on $\fB_n$ and satisfies $(\dag)$. 
Moreover, the quotient $\fB_n/\G$ is finite. 
For these facts we refer to \cite[p. 140]{Laumon}, \cite{Li}, \cite{LSV}, \cite{Serre}.
Theorem
\ref{thm3.1} implies $H^i(\fB_n/\G)=0$ for all $1\leq i\leq n$. On
the contrary, $H^{n+1}(\fB_n/\G)$ is usually quite large. Its
dimension approximately equals the volume of $\PGL_{n+2}(K)/\G$
with respect to an appropriately normalized Haar measure on
$\PGL_{n+2}(K)$; see \cite{Serre}. The simplicial complexes $\fB_n/\G$ are often used 
in the construction of Ramanujan complexes; see
\cite{Li}, \cite{LSV}.

\subsection{Property (T)} Garland's method has been applied to prove that 
certain groups have Kazhdan's property (T).  

Let $\G$ be a group generated by a finite set $S$. Let $\pi:\G\to U(H_\pi)$ be a 
unitary representation. We say that $\pi$ almost has invariant vectors if for every $\eps>0$ 
there exists a non-zero vector $u_\eps$ in the Hilbert space $H_\pi$ such that 
$|\!|\pi(s)u_\eps-u_\eps|\!|\leq \eps |\!|u_\eps|\!|$ 
for every $s\in S$. 
The group $\G$ is said to have \textit{property (T)} if 
every unitary representation of $\G$ which almost has invariant vectors has a non-zero invariant vector.  

Property (T) has important applications to representation theory, 
ergodic theory, geometric group theory and the theory of networks. 
For example, Margulis used groups with property (T) to give the first explicit examples of expanding graphs 
and to solve the Banach-Ruziewicz problem that asks whether the Lebesgue measure is the only normalized rotationally invariant 
finitely additive measure on the $n$-dimensional sphere.
We refer to Lubotzky's book \cite{Lubotzky} for a discussion of property (T) and its applications. 

It is known that a group $\G$ has property (T) if and only if for any unitary representation $\pi$ of $\G$, the 
first cohomology group $H^1(\G, \pi)$ is zero. This suggests the following line of attack 
to prove that $\G$ has property (T). 
Suppose that $\G$ is the fundamental group of a finite simplicial complex $X$. 
By group cohomology, $H^1(\G, \pi)=H^1(X, E_\pi)$, where $E_\pi$ is a local system on $X$
associated to $\pi$. Then one can try to prove the vanishing of $H^1(X, E_\pi)$ by a 
generalization of Garland's method. This approach in the case when $X$ is a $2$-dimensional finite simplicial complex
was pursued independently by Ballmann and {\'S}wiatkowski \cite{BS}, Pansu \cite{Pansu}, and \.Zuk \cite{Zuk}. 
For example, in \cite{BS}, the authors prove the following theorem: Assume $X$ is a 
$2$-dimensional finite simplicial complex, $\Lk(v)$ is a connected graph for any vertex of $X$, and $\la^0_{\min}(X)>1/2$. 
Then $\G=\pi_1(X)$ has property (T). Note that these assumptions are the same as in Corollary \ref{CorFI} 
for $n=2$. They are fulfilled when $X$ is a finite quotient of a 2-dimensional Bruhat-Tits building. 
These results gave new explicit examples of groups with property (T) which 
were significantly different from the earlier known examples. 
In \cite{DJ1}, \cite{DJ2}, Dymara and Januszkiewicz applied a generalization of Garland's method to groups acting on buildings of arbitrary type 
and dimension (e.g. hyperbolic buildings), and produced examples of groups having property (T), not coming from locally symmetric spaces or euclidean
buildings.


\section{Complexes of flags}\label{sec3}

The main goal of this section is to prove Theorem \ref{thm-G}. 
The notation will be the same as in $\S$\ref{ssCF}. 
In particular, $V$ is a linear space of dimension $n+2$ over the finite field $\F_q$, and $\cF$ 
is a (possibly empty) flag in $V$ of length $\ell$. We denote 
$$
N=\dim(X_\cF). 
$$

The proof of Theorem \ref{thm-G} proceeds by induction on $N$ and $i$. The base case $N=1$
follows from a direct calculation. We will carry out this calculation in $\S$\ref{ss4.1}. 
In the same subsection we give some explicit 
examples which provide a sense of the complexity of the eigenvalues of $\Delta$ acting on $C^i(X_\cF)$. 
These examples suggest a remarkable asymptotic behaviour of the eigenvalues of $\Delta$ as $q\to \infty$, 
which we state as a conjecture. 

The inductive step, discussed in $\S$\ref{ss4.2}, has two parts. 
Assuming the claim holds for $i=0$ and all $N$, the proof of the general case 
quickly follows from the inequality in Theorem \ref{thmFIeigen}. 
On the other hand, the argument which proves the claim for $i=0$ and $N\geq 1$ is fairly intricate. 
The outline is approximately the following: We start with a $\Delta$-eigenfunction $f\in C^0(X_\cF)$ 
having eigenvalue $c>0$. The machinery developed in $\S$\ref{ss3.2} cannot be applied to this function, since 
we cannot apply the operator $\tau_v$ directly to $f$. Instead, we introduce a parameter $R\in \R$, 
and multiply the values of $f$ on an appropriate subset of $\Ver(X_\cF)$ by $R$. The resulting 
function $f_\alpha$ is no longer an eigenfunction of $\Delta$, but we get some 
flexibility because we can vary $R$. We apply 
the machinery of $\S$\ref{ss3.2} to $df_\alpha\in C^1(X_\cF)$. 
Choosing $R$ appropriately forces some miraculous cancellations, 
which in the end give the desired bound $c\geq N-\eps$. 

In $\S$\ref{ss4.3}, we prove some auxiliary results about the eigenvalues of 
curvature transformations. These results are not used elsewhere in the paper, and are given 
as some evidence for the conjecture in $\S$\ref{ss4.1}. 

\subsection{The base case and explicit examples}\label{ss4.1} For $N=1$ we 
need to consider only $\Delta$ acting on $C^0(X_\cF)$, since $0\leq i\leq N-1$. 

\begin{lem}\label{lem-n2.2} If $N=1$, then 
$m^0(X_\cF)$ is equal either to $1$ or $1-\frac{\sqrt{q}}{q+1}$.
\end{lem}
\begin{proof} If $\dim(X_\cF)=1$ then the length of $\cF$ is $\ell=n-2$. Let $(t_0,\dots,
t_{n-1})$ be defined by \eqref{eq-t}. Since $t_i\geq 1$ and $\sum_{i=0}^{n-1} (t_i-1)=2$,
either exactly two $t_i$, $t_j$, $i<j$, are equal to $2$ and all
others are $1$, or exactly one $t_i$ is equal to $3$ and all others
are $1$. In the first case $X_\cF\cong X_\varnothing^0\ast
X_\varnothing^0$, in the second case $X_\cF\cong X_\varnothing^1$.

In the first case $X_\cF$ is a $(q+1)$-regular bipartite graph with $2(q+1)$ vertices. 
It is easy to check that $(q+1)\Delta$ acts on $C^0(X_\cF)$ as the
matrix
$$
(q+1)I_{2(q+1)}-\begin{pmatrix} 0 & J_{q+1} \\ J_{q+1} & 0\end{pmatrix}. 
$$
The minimal polynomial of this matrix is $x(x-(q+1))(x-2(q+1))$, so the
eigenvalues of $\Delta$ are $0$, $1$, and $2$.

In the second case, $X_\cF$ is isomorphic to the graph 
whose vertices correspond to $1$ and $2$-dimensional
subspaces of a $3$-dimensional vector space $V$ over $\F_q$, two
vertices being adjacent if one of the corresponding subspaces is
contained in the other. With a slight abuse of terminology, we will
call $1$ and $2$ dimensional subspaces lines and planes,
respectively. The number of lines and planes in $V$ is $m=q^2+q+1$
each. Let $A=(a_{ij})$ be the $m\times m$ matrix whose rows are
enumerated by the lines in $V$ and columns by the planes, and
$a_{ij}=-1$ if the $i$th line lies in the $j$th plane, and is $0$
otherwise. We can choose a basis of $C^0(X_\cF)$ so that
$(q+1)\Delta$ acts as the matrix
$$
(q+1)I_{2m}+\begin{pmatrix} 0 & A \\ A^t & 0\end{pmatrix},
$$
where $A^t$ denotes the transpose of $A$. 
Let $M=\begin{pmatrix} 0 & A \\ A^t & 0\end{pmatrix}$. Since any two
distinct lines lie in a unique plane and any line lies in $(q+1)$
planes, $AA^t=qI_m+J_m$. By a similar argument, $A^tA=qI_m+J_m$.
Hence
$$
M^2= qI_{2m}+\begin{pmatrix} J_m & 0 \\ 0 & J_m\end{pmatrix}.
$$
This implies that $(M^2-qI_{2m})(M^2-(q+1)^2I_{2m})=0$. Since
$(q+1)\Delta - (q+1)I_{2m}=M$, we conclude that $(q+1)\Delta$
satisfies the polynomial equation
$$
x(x-(2q+2))(x^2-(2q+2)x+(q^2+q+1))=0.
$$
It is not hard to see that this is in fact the minimal polynomial of
$(q+1)\Delta$. Hence the eigenvalues of $\Delta$ are $0$, $2$, and
$1\pm\frac{\sqrt{q}}{q+1}$.
\end{proof}

Denote by $\mathrm{min.pol}^i_n(x)$ the minimal polynomial of $\Delta$ acting on
$C^i(X^n_\varnothing)$. The proof of Lemma \ref{lem-n2.2} shows that  
$$
\mathrm{min.pol}^0_1(x)=x(x-2)\left(x^2-2x+\frac{q^2+q+1}{q^2+2q+1}\right).
$$
Note that $$m^0(X^1_\varnothing)=1- \frac{\sqrt{q}}{q+1}$$
is always strictly larger than $1/2$ and tends to $1$ as 
$q\to \infty$. Moreover, the whole polynomial tends coefficientwise to the
polynomial $x(x-2)(x-1)^2$.

Now assume $n=2$. In this case it is considerably harder to compute
the minimal polynomials. With the help of a computer, we deduced that
\begin{align*}
\mathrm{min.pol}^0_2(x)=&x(x-2)\left(x-3\right)\left(x-\frac{2q^2+3q+2}{q^2+q+1}\right)\\
&\times\left(x^2-\frac{4q^2+3q+4}{q^2+q+1}x+\frac{4q^2+4}{q^2+q+1}\right).
\end{align*}
This implies 
$$
m^0(X^2_\varnothing)= \frac{1}{2(q^2+q+1)}\left(4q^2+3q+4-\sqrt{8q^3+9q^2+8q}\right)
$$
is at least $1.08$ and tends to $2$ from below as $q\to
\infty$. The whole polynomial tends coefficientwise to the
polynomial $x(x-3)(x-2)^4$ as $q\to \infty$. Next
\begin{align*}
\mathrm{min.pol}^1_2(x)=&x(x-1)(x-2)(x-3)\\
&\times\left(x^2-2x+\frac{q^2+1}{q^2+2q+1}\right)
\left(x^2-3x+\frac{2q^2+2q+2}{q^2+2q+1}\right)\\
&\times\left(x^2-4x+\frac{4q^2+6q+4}{q^2+2q+1}\right).
\end{align*}
In this case $$m^1(X^2_\varnothing)=1-\frac{\sqrt{2q}}{q+1}.$$ 
It is easy to see that $1/3\leq m^1(X^2_\varnothing)<1$.  Moreover, $m^1(X^2_\varnothing)$ is strictly larger
than $1/3$ for $q>2$ and tends to $1$ as $q\to \infty$; the whole
polynomial tends to $x(x-3)(x-2)^4(x-1)^4$.

\begin{conj}\label{conj} 
The previous examples, combined with some calculations for $n=3$ 
which we do not list, suggest a remarkable property of the eigenvalues of $\Delta$ 
acting on $C^i(X^n_\varnothing)$, $0\leq i\leq n-1$:
\begin{enumerate}
\item The number of distinct eigenvalues of $\Delta$ depends only on $i$, i.e., does not depend on $q$, even though the
eigenvalues themselves and the dimension of $C^i(X^n_\varnothing)$ depend on $q$. 
\item The positive eigenvalues of $\Delta$, which in general are neither rational nor integral, 
tend to the integers
$$
n-i,\ n-i+1,\ \dots,\ n+1
$$ as $q\to \infty$.
\end{enumerate}
\end{conj}

\subsection{Inductive step}\label{ss4.2} Since we proved Theorem \ref{thm-G} 
for $N=1$, we assume $N\geq 2$. Let $1\leq i\leq N-1$ be given. 
Assume for the moment that we proved the bound in Theorem \ref{thm-G} for $\Delta$ 
acting on $C^{i-1}(X_\cG)$, where $\cG$ is any flag with $\dim(\cG)= N-1$.  
Since for any $v\in \Ver(X_\cF)$ its link $\Lk(v)$ 
is isomorphic to $X_\cG$ for some $\cG\succ\cF$ with $\dim(X_\cG)=N-1$, 
we get
\begin{align*}
\la^{i-1}_{\min}(X_\cF) &\geq (N-1)-(i-1)-\eps'=N-i-\eps',
\end{align*}
where $\eps'=i\cdot\eps/(i+1)$. Then, by Theorem \ref{thmFIeigen}, we have  
\begin{equation}\label{eqNNN2}
m^i(X_\cF) \geq \frac{(i+1)\cdot \la^{i-1}_{\min}(X_\cF)-(N-i)}{i} \geq  N-i-\eps. 
\end{equation}

Therefore, to complete the proof of Theorem \ref{thm-G} it remains to show 
that 
\begin{equation}\label{eqNNN}
m^0(X_\cF)\geq N-\eps.
\end{equation} This will occupy the rest of this subsection.

\begin{rem}
Instead of induction, one can deduce the lower bound \eqref{eqNNN2} 
directly from \eqref{eqNNN} using Corollary \ref{corFIeigen}. Indeed, the link of any 
$(i-1)$-simplex in $X_\cF$ is isomorphic to some $X_\cG$ with $\dim(X_\cG)=N-i$, so 
assuming $m^0(X_\cG)\geq (N-i)-\eps'$, $\eps'=\eps/(i+1)$, Corollary \ref{corFIeigen} gives 
$$
m^i(X_\cF)\geq (i+1)(N-i-\eps')-i(N-i)= N-i-\eps
$$
On the other hand, 
the proof of Corollary \ref{corFIeigen} uses similar inductive argument as above. 
\end{rem}

We start by proving some preliminary lemmas.  For an integer $m\geq 1$ we put $(m)_q=\prod_{k=1}^m(q^k-1)$, and we put $(0)_q=1$. 
The number of $d$-dimensional subspaces in an $m$-dimensional linear space over $\F_q$ is equal to 
$$
\gc{m}{d}:=\frac{(m)_q}{(d)_q(m-d)_q}.
$$
With this notation it is easy to give a formula for the number of $n$-flags refining a given flag:
\begin{lem}\label{lem-3.1} 
Let $s$ be a simplex in $X_\cF$ corresponding to $\cG\succ \cF$. Let $(r_0,\dots, r_{j})$ 
be the integers defined for $\cG$ by \eqref{eq-t}. The number of $N$-simplices in $X_\cF$
containing $s$ is given by the formula
$$
w(s)=\prod_{k=0}^j\prod_{z=1}^{r_k}\gc{z}{1}=\prod_{k=0}^{j}(r_k)_q/(1)^{r_k}_q. 
$$
\end{lem}

Let $v\in \Ver(X_\cF)$ and let $\cG$ be the
corresponding $(\ell+1)$-flag. There is a unique subspace $G$ in the
sequence of $\cG$ which does not occur in $\cF$. Let
$$\Type(v):=\dim(G).$$ 
Denote the set of types of vertices of $X_\cF$ by $\fT$. 
It is easy to see that the vertices of a simplex in $X_\cF$ 
have distinct types. Moreover, $\# \fT=N+1$. 

\begin{lem}\label{eq-yet} Let $v\in \Ver(X_\cF)$. Assume $\alpha\in \fT$ is fixed and $\alpha\neq \Type(v)$.  Then 
$$
\sum_{\substack{x\in \Ver(X_\cF)\\ [v, x]\in \widehat{S}_{1}(X_\cF) \\ \Type(x)=\alpha}}w([v,x]) = w(v). 
$$
\end{lem}
\begin{proof}
Let $\cG$ be the flag of length $i:=\ell+1$ in $V$ corresponding
to $v$. Let $(t_0,\dots, t_{i+1})$ be the array \eqref{eq-t} of $\cG$. Let $[v, x]\in \widehat{S}_{1}(X_\cF)$ 
and $\cG'\succ \cG$ be the corresponding $(i+1)$-flag. There is a unique $t_a$
such that the array of $\cG'$ is $(t_0,\dots,t_a', t_a'',\cdots, t_{i+1})$
with $t_a'+t_a''=t_a$. Moreover, the type of $x$ uniquely determines  
$a$ and $t_a'$. The number of $[v, x]\in \widehat{S}_{1}(X_\cF)$ with $\Type(x)=\alpha$ is
equal to $\gc{t_a}{t_a'}$. Using Lemma \ref{lem-3.1}, we compute 
$$
\sum_{\substack{x\in \Ver(X_\cF)\\ [v, x]\in \widehat{S}_{1}(X_\cF) \\ \Type(x)=\alpha}}\frac{w([v,x])}{w(v)}
=\gc{t_a}{t_a'}\frac{(1)_q^{t_a}
(t_a')_q(t_a'')_q}{(t_a)_q(1)_q^{t_a'}(1)_q^{t_a''}}=1.
$$
\end{proof}

\begin{rem} Lemma \ref{eq-yet} is a refined version of Lemma \ref{lem-w}. Indeed, 
$$
\sum_{\substack{x\in \Ver(X_\cF)\\ [v, x]\in \widehat{S}_{1}(X_\cF) }}w([v,x]) =
\sum_{\substack{\alpha\in \fT\\ \alpha\neq \Type(v)}}
\sum_{\substack{x\in \Ver(X_\cF)\\ [v, x]\in \widehat{S}_{1}(X_\cF) \\ \Type(x)=\alpha}}w([v,x]) = 
\sum_{\substack{\alpha\in \fT\\ \alpha\neq \Type(v)}} w(v) = Nw(v).  
$$ 
\end{rem}

Let $f\in C^0(X_\cF)$ and let $R\in \R$ be a fixed constant. For each 
$\alpha\in \fT$ define the function $f_\alpha\in C^0(X_\cF)$ by 
$$
f_\alpha(v) = 
\begin{cases}
R\cdot f(v), & \text{if $\Type(v)= \alpha$};\\ 
f(v), & \text{if $\Type(v)\neq \alpha$}.
\end{cases}
$$
Also, for $i\geq 0$ define a linear transformation $\rho_\alpha: C^i(X_\cF)\to C^i(X_\cF)$ by
$$
\rho_\alpha=\sum_{\substack{v\in \Ver(X_\cF)\\ \Type(v)=\alpha}}\rho_v.
$$

\begin{lem} We have 
\begin{align}
\label{eq-ny3} 
\sum_{\alpha\in \fT}(1-\rho_\alpha)df&=(N-1)df, \\
\label{eq-ny} 
\sum_{\substack{v\in \Ver(X_\cF)\\ \Type(v)=\alpha}} (\Delta\rho_v df_\alpha,
\rho_vdf_\alpha) &=(\Delta \rho_\alpha df_\alpha, \rho_\alpha df_\alpha),  \\ 
\label{lem2.3}
(\rho_\alpha df_\alpha, df_\alpha) & =(df_\alpha,
df_\alpha)-((1-\rho_\alpha)df, df),\\ 
\label{lem2.2}
(\Delta\rho_\alpha df_\alpha, \rho_\alpha df_\alpha) &=((1-\rho_\alpha)df,df).
\end{align}
\end{lem}
\begin{proof} Equation \eqref{eq-ny3} follows from a straightforward calculation:  
$$
\sum_{\alpha\in \fT}(1-\rho_\alpha)df=(N+1)df-\sum_{v\in
\Ver(X_\cF)}\rho_v df \nonumber =(N+1)df-2df=(N-1)df. 
$$

To prove \eqref{eq-ny}, expand its right hand-side as 
$$
(d \rho_\alpha df_\alpha, d\rho_\alpha df_\alpha)=\sum_{\substack{v, v'\in \Ver(X_\cF)\\ \Type(v)=\Type(v)=\alpha}} 
(d\rho_v df_\alpha, d\rho_{v'} df_\alpha). 
$$
Now let $s=[x,y,z]\in S_2(X_\cF)$. Since the vertices of the same simplex have distinct types, only one of $x,y,z$ can be of type $\alpha$. 
Therefore, if $v\neq v'$ but $\Type(v)=\Type(v)=\alpha$, then $d\rho_v df_\alpha(s)\cdot d\rho_{v'} df_\alpha(s)=0$. 
This implies that in the above sum only the terms with $v=v'$ are possibly non-zero, so  
$$
\sum_{\substack{v, v'\in \Ver(X_\cF)\\ \Type(v)=\Type(v)=\alpha}}  (d\rho_v df_\alpha, d\rho_{v'} df_\alpha) = 
\sum_{\substack{v \in \Ver(X_\cF)\\ \Type(v)=\alpha}}  (d\rho_v df_\alpha, d\rho_v df_\alpha). 
$$

To prove \eqref{lem2.3}, note that if $s\in S_1(X_\cF)$ contains a vertex of type
$\alpha$, then $(1-\rho_\alpha)g(s)=0$ for any $g\in C^1(X)$. 
On the other hand, if $s$ does not contain a vertex of type $\alpha$, then 
$(1-\rho_\alpha)df_\alpha(s)= df_\alpha(s)=df(s)=(1-\rho_\alpha)df(s)$. Hence
$$((1-\rho_\alpha)df_\alpha, df_\alpha)=((1-\rho_\alpha)df, df).$$ Now
$$
((1-\rho_\alpha)df, df)=((1-\rho_\alpha)df_\alpha,
df_\alpha)=(df_\alpha, df_\alpha)-(\rho_\alpha df_\alpha,
df_\alpha).
$$

Finally, to prove \eqref{lem2.2}, let $s=[x,y,z]\in S_2(X_\cF)$. If none of the vertices
of $s$ has type $\alpha$ then $d\rho_\alpha d f_\alpha(s)=0$. If $s$
has a vertex of type $\alpha$, then such a vertex is unique. Without
loss of generality, assume $\Type(x)=\alpha$. Then
$$
d\rho_\alpha d f_\alpha([x,y,z])=f(y)-f(z)=-df([y,z]).
$$
Hence
$$
(d\rho_\alpha d f_\alpha, d\rho_\alpha d f_\alpha)=\sum_{\substack{v\in \Ver(X_\cF)\\
\Type(v)=\alpha}}\sum_{s\in \widehat{S}_1(\Lk(v))}w([v,s])df(s)^2
$$
$$
=\sum_{s\in \widehat{S}_1(X_\cF)} (1-\rho_\alpha)df(s)\cdot df(s)
\sum_{\substack{v\in \Ver(\Lk(s))\\
\Type(v)=\alpha}}w([v,s])= ((1-\rho_\alpha)df, df), 
$$
where in the last equality we used Lemma \ref{eq-yet}.
\end{proof}

\begin{lem}\label{lemd31}
Let $f\in C^0(X_\cF)$ and suppose $\Delta f=c\cdot f$. Then
$$
\sum_{\alpha\in \fT}(\Delta f_\alpha,
f_\alpha)=\left[(N-c)(R-1)^2+c(R^2+N)\right]\cdot (f,f).
$$
\end{lem}
\begin{proof} Fix some type $\alpha$ and
let $g\in C^0(X_\cF)$ be a function such that $g(v)=0$ if
$\Type(v)\neq \alpha$. Then $(\Delta g, g)=N\cdot (g,g)$. Indeed,
\begin{align*}
(\Delta g, g) &=(dg,dg)=\sum_{[x,v]\in \widehat{S}_1(X_\cF)}w([x,v])(g(v)-g(x))^2 \\
& =\sum_{\substack{v\in \Ver(X_\cF)\\ \Type(v)=\alpha}}g(v)^2\sum_{\substack{x\in \Ver(X_\cF)\\ [x,v]\in \widehat{S}_1(X_\cF)}}w([x,v]) \\
&\overset{\mathrm{Lem.} \ref{lem-w}}{=}N\sum_{\substack{v\in \Ver(X_\cF)\\ \Type(v)=\alpha}}w(v)\cdot g(v)^2=N\cdot
(g,g).
\end{align*}
If we apply this to $g=f_\alpha-f$, then we get
\begin{equation}\label{eq-d31}
(\Delta f_\alpha, f_\alpha)=N\cdot (f_\alpha,
f_\alpha)-2(N-c)(f_\alpha, f)+(N-c)(f,f).
\end{equation}
Since the cardinality of $\fT$ is $(N+1)$,
$$
\sum_{\alpha\in \fT}f_\alpha= (N+R)\cdot f\quad \text{and} \quad
\sum_{\alpha\in \fT}(f_\alpha, f_\alpha)=(N+R^2)\cdot (f,f).
$$
Summing \eqref{eq-d31} over all types and using the previous two
equalities, we get the claim.
\end{proof}

\begin{prop}\label{prop4.15-15} For any $\eps>0$ there is a constant $q(\eps, n)$ depending only on
$\eps$ and $n$, such that if $q>q(\eps, n)$ then $m^0(X_\cF)\geq
N-\eps$.
\end{prop}
\begin{proof} Since Lemma \ref{lem-n2.2} implies this claim for $N=1$, we can assume 
from now on that $N\geq 2$.

Let $f\in C^0(X_\cF)$ and suppose $\Delta
f=c\cdot f$. If $\Type(v)=\alpha$, then
\begin{align}
\Delta f_\alpha(v) = \sum_{\substack{x\in \Ver(X_\cF)\\ [x,v]\in S_1(X_\cF)}}\frac{w([x,v])}{w(v)}(Rf(v)-f(x))
=NRf(v)-C,
\end{align}
where $C$ does not depend on $R$ since $\Type(x)\neq \Type(v)$ if $[x,v]\in S_1(X_\cF)$. 
If we take $R=1$, then $f_\alpha=f$, so $\Delta f_\alpha(v)=c\cdot f(v)$. 
We conclude that $C=(N-c)f(v)$, and 
$\Delta f_\alpha(v)=(NR-(N-c))f(v)$. 
From now on we assume that $R=(N-c)/N$. With this choice of $R$ our calculation implies 
\begin{equation}\label{eq-deltaalpha}
\Delta f_\alpha(v)=0\quad \text{if } \Type(v)=\alpha.  
\end{equation}

Let $v\in \Ver(X_\cF)$ be a vertex of type $\alpha$. By Lemma \ref{prop7.14}, 
$$
(\Delta \rho_v df_\alpha, \rho_vdf_\alpha)= (\Delta_v\tau_v df_\alpha, \tau_v df_\alpha)_v. 
$$
Since 
$$
(\mathbf{1}, \tau_v df_\alpha)_v \overset{\eqref{eq-1tuaf}}{=} -w(v)\delta df_\alpha(v)= -w(v)\Delta f_\alpha(v) \overset{\eqref{eq-deltaalpha}}{=} 0, 
$$
we can use the argument in the proof of Lemma \ref{lemDec5} to conclude  
$$
(\Delta_v\tau_v df_\alpha, \tau_v df_\alpha)_v\geq m^{0}(\Lk(v))(\tau_v df_\alpha, \tau_v df_\alpha)_v
\geq \la^0_{\min}(X_\cF) (\tau_v df_\alpha, \tau_v df_\alpha)_v. 
$$
(Note that $\Lk(v)$ is connected since $\Lk(v)\cong X_\cG$ for some $\cG$ with $\dim(X_\cG)=N-1\geq 1$.)
Hence 
\begin{align*}
(\Delta \rho_v df_\alpha, \rho_vdf_\alpha)\geq \la^0_{\min}(X_\cF) (\tau_v df_\alpha, \tau_v df_\alpha)_v 
 &\overset{\mathrm{Lem}. \ref{prop7.12}}{=} \la^0_{\min}(X_\cF) (\rho_v df_\alpha, \rho_v df_\alpha) \\ 
& \overset{\eqref{eq-(rho)}}{=}\la^0_{\min}(X_\cF) (\rho_v df_\alpha, df_\alpha). 
\end{align*}
Summing these inequalities over all vertices of type $\alpha$ and using \eqref{eq-ny}, we get 
$$
(\Delta \rho_\alpha df_\alpha, \rho_\alpha df_\alpha)\geq
\la^0_{\min}(X_\cF)\cdot (\rho_\alpha df_\alpha, df_\alpha).
$$
Using \eqref{lem2.3}  and \eqref{lem2.2}, we can rewrite this inequality as 
$$
(1+\la^0_{\min}(X_\cF))((1-\rho_\alpha)df,df)\geq
\la^0_{\min}(X_\cF)\cdot (df_\alpha, df_\alpha)
$$
Summing these inequalities over all types and using \eqref{eq-ny3}
and Lemma \ref{lemd31}, we get
\begin{equation}\label{eq2.1}
(1+\la^0_{\min}(X_\cF))(N-1)c\geq \la^0_{\min}(X_\cF)\cdot
\left[(N-c)(R-1)^2+c(R^2+N)\right].
\end{equation}

Suppose $c=m^0(X_\cF)$. If $c\geq N$, then we are done. On the other hand, if $c<N$, then 
$(N-c)(R-1)^2$ is positive, so \eqref{eq2.1} implies 
$$
(1+\la^0_{\min}(X_\cF))(N-1)c \geq \la^0_{\min}(X_\cF)c (R^2+N).
$$
Dividing both sides by $c$ (recall that $c>0$), we get 
$$
N-1\geq (1+R^2)\la^0_{\min}(X_\cF). 
$$
By induction on $N$, for any $\eps>0$ there is a constant $q(\eps,
n)$ such that $\la^0_{\min}(X_\cF)\geq N-1-\eps$ if $q\geq q(\eps,
n)$. Thus 
$$
\eps\geq R^2(N-1-\eps). 
$$
We see that $R^2\to 0$ as $\eps\to 0$. Since $R=(N-c)/N$, this forces $c\to N$. 
\end{proof}

\subsection{Auxiliary results about eigenvalues}\label{ss4.3}
In this subsection we prove that $M^i(X_\cF)\leq N+1$ and $m^0(X_\cF)\leq N$. This 
implies that if we allow $q$ to vary, then the lower bound $m^0(X_\cF)\geq N-\eps$ in Proposition \ref{prop4.15-15} 
is optimal; in other terms, $m^0(X_\cF)\to N$ as $q\to \infty$, which is consistent with Conjecture \ref{conj}. 
We also show that $M^0(X_\cF)=N+1$ and its multiplicity is $N$, so does not depend on $q$.  

\begin{prop}\label{prop_ny} For all $0\leq i\leq N-1$ we have 
$M^i(X_\cF)\leq N+1$.
\end{prop}
\begin{proof}  The proof is again by induction on $N$. If $N=1$, then the calculations in the proof 
of Lemma \ref{lem-n2.2} show that $M^0(X_\cF)=2$. 
Now assume $N\geq 2$ and $i\geq 1$. Assume we proved that $M^{i-1}(X_\cG)\leq N$ 
for any $\cG$ with $\dim(X_\cG)=N-1$. Then $\la_{\max}^{i-1}(X_\cF)\leq N$, so Theorem \ref{thmFIeigen} 
implies $M^i(X_\cF)\leq N+1$. It remains to prove that $M^0(X_\cF)\leq N+1$. 

Let $f\in C^0(X_\cF)$.  By an argument very similar to the proof of Proposition \ref{prop4.15-15} we get 
\begin{align*}
(\Delta \rho_v df_\alpha, \rho_vdf_\alpha)  = (\Delta_v\tau_v df_\alpha, \tau_v df_\alpha)_v
&\leq \la^0_{\max}(X_\cF)(\tau_v df_\alpha, \tau_v df_\alpha)_v \\ 
& = \la^0_{\max}(X_\cF)\cdot (\rho_v df_\alpha, df_\alpha),
\end{align*}
which leads to 
$$
(\Delta \rho_\alpha df_\alpha, \rho_\alpha df_\alpha)\leq \la^0_{\max}(X_\cF)\cdot (\rho_\alpha df_\alpha, df_\alpha). 
$$
By induction, $\la^0_{\max}(X_\cF)\leq N$, so using \eqref{lem2.3} and \eqref{lem2.2}, we can rewrite the previous inequality as 
$$
(1+N)\cdot ((1-\rho_\alpha)df,df)\leq N\cdot (df_\alpha, df_\alpha).
$$
Assume $\Delta f=c\cdot f$ is an eigenfunction. Summing the above inequalities over all types and using \eqref{eq-ny3}
and Lemma \ref{lemd31}, we get
$$
(N+1)(N-1)c\leq N\cdot
\left[(N-c)(R-1)^2+c(R^2+N)\right].
$$
If we put $R=(N-c)/N$, then this inequality forces $c\leq
N+1$. In particular, $M^0(X_\cF)\leq N+1$.
\end{proof}

Let $Y$ be an $N$-simplex. Since $Y$ has a unique simplex of maximal dimension, the weights \eqref{eq-themetric} of  
the simplices of $Y$ are all equal to $1$. Then, relative to the inner-product \eqref{eq-pairing}, we have the orthogonal direct sum decomposition 
(cf. Lemma \ref{lem1.7})
$$
C^0(Y)=\R\mathbf{1}\oplus \delta C^1(Y),
$$
where $\delta C^1(Y)$ can be explicitly described as the space of functions satisfying $$\sum_{v\in \Ver(Y)}g(v)=0.$$  
It is easy to check that $\Delta g = 0$ if and only if $g\in \R\mathbf{1}$, and $\Delta g= (N+1)g$ if and only if $g\in \delta C^1(Y)$. 
Hence $0$ and $N+1$ are the only eigenvalues of $\Delta$ acting on $C^0(Y)$, and their multiplicities are $1$ and $N$, 
respectively. 

\begin{defn} We say that $f\in C^0(X_\cF)$ is \textit{type-constant} if $f(v)=f(v')$ for all $v, v'\in \Ver(X_\cF)$ of the same type. 
We denote the space of type-constant functions by $\sC$. 
\end{defn}
Label the vertices of $Y$ by the elements of $\fT$. 
Given a function $f\in \sC$, define $\tilde{f}\in C^0(Y)$ by $\tilde{f}(x)=c_{\alpha}(f)$, $x\in \Ver(Y)$, where 
$\alpha=\Type(x)$ and $c_\alpha(f)$ is the value of $f$ on vertices of type $\alpha\in \fT$. 
It is clear that $\sC\to C^0(Y)$, $f\mapsto \tilde{f}$, is an isomorphism of vector spaces which restricts 
to an isomorphism $\sC_0 \xrightarrow{\sim} \delta C^1(Y)$, where $\sC_0\subset \sC$ is the subspace of functions $f\in \sC$ satisfying 
$$
\sum_{\alpha\in \fT}c_\alpha(f)=0. 
$$

\begin{lem}\label{lemNNC}
If $f\in \sC$, then $\Delta f\in \sC$. Moreover,  $\widetilde{\Delta f} = \Delta \tilde{f}$. 
This implies that for $f\in \sC_0$ we have $\Delta f= (N+1)f$. 
\end{lem}
\begin{proof}
For a fixed $v\in \Ver(X_\cF)$, we have
$$
\Delta f(v)=\sum_{\substack{x\in \Ver(X_\cF)\\ [x,v]\in S_1(X_\cF)}}\frac{w([x,v])}{w(v)}(f(v)-f(x))
$$
$$
\overset{\mathrm{Lem.} \ref{lem-w}}{=}N f(v)- \sum_{\substack{x\in \Ver(X_\cF)\\ [x,v]\in S_1(X_\cF)}}\frac{w([x,v])}{w(v)}f(x)
$$
$$
= N f(v)- \sum_{\substack{\alpha\in \fT\\ \alpha\neq \Type(v)}}
\sum_{\substack{x\in \Ver(X_\cF)\\ \Type(x)=\alpha \\ [x,v]\in S_1(X_\cF)}}\frac{w([x,v])}{w(v)}f(x).  
$$
Now
$$
\sum_{\substack{x\in \Ver(X_\cF)\\ \Type(x)=\alpha \\ [x,v]\in S_1(X_\cF)}}\frac{w([x,v])}{w(v)}f(x) = 
c_\alpha(f)\sum_{\substack{x\in \Ver(X_\cF)\\ \Type(x)
=\alpha \\ [x,v]\in S_1(X_\cF)}}\frac{w([x,v])}{w(v)}\overset{\mathrm{Lem.}\ref{eq-yet}}{=}  c_\alpha(f). 
$$
Thus, if we denote $\beta=\Type(v)$, 
$$
\Delta f(v) = N c_\beta(f) - \sum_{\substack{\alpha\in \fT\\ \alpha\neq \beta}} c_\alpha(f). 
$$
It is clear from this that $\Delta f\in \sC$. Moreover, for any $z\in \Ver(Y)$ we have 
$$
\widetilde{\Delta f}(z)=N\tilde{f}(z)-\sum_{\substack{y\in \Ver(Y_\cF)\\ y\neq z}}
\tilde{f}(y)=\Delta\tilde{f}(z).  
$$
\end{proof}

\begin{lem}\label{lemNNC2}
Let $f\in C^0(X_\cF)$ be a $\Delta$-eigenfunction with eigenvalue $c$. If $c=0$ or $N+1$, then 
$f\in \sC$. 
\end{lem}
\begin{proof}
Using the fact that $X_\cF$ is connected, it is easy to show that $\Delta f=0$ if and only if $f$ 
is constant.  We prove that if $\Delta f=(N+1)f$, then $f$ is type-constant.  

First, assume $N=1$. Then either $X_\cF=X_\varnothing^0\ast X_\varnothing^0$ or $X_\cF=X_\varnothing^1$. 
In either case, the matrix of $m\Delta$ has the form 
$$
\begin{pmatrix} 
m I_m & -A \\
-A^t & m I_m
\end{pmatrix},
$$ 
where $A$ gives the adjacency relations between vertices of type $\alpha$ and $\beta$ (there are only two types), and 
either $m=q+1$ or $m=q^2+q+1$. Let $\mathbf{x}=(x_1, x_2, \dots, x_{2m})^t$ be an eigenvector with eigenvalue $2m$. 
Let $\mathbf{x}_\alpha=(x_1, x_2, \dots, x_{m})^t$ and $\mathbf{x}_\beta=(x_{m+1}, x_2, \dots, x_{2m})^t$. 
Then 
\begin{align*}
mI_m \mathbf{x}_\alpha - A\mathbf{x}_\beta &= 2m \mathbf{x}_\alpha \\
-A^t \mathbf{x}_\alpha + mI_m \mathbf{x}_\beta &= 2m \mathbf{x}_\beta.
\end{align*}
Hence $m \mathbf{x}_\alpha = -A\mathbf{x}_\beta$ and $m \mathbf{x}_\beta =-A^t \mathbf{x}_\alpha$. 
This implies $m^2 \mathbf{x}_\alpha = AA^t \mathbf{x}_\alpha$. In the first case, $AA^t=m J_m$, 
so $m \mathbf{x}_\alpha = J_m \mathbf{x}_\alpha$. This implies that $m x_j=\sum_{i=1}^m x_i$ for all $1\leq j\leq m$. 
Hence $x_1=x_2=\dots=x_m$. Similarly, one shows that $x_{m+1}=\dots=x_{2m}$. In the second case, $AA^t=qI_m+J_m$, so 
$$
m^2 \mathbf{x}_\alpha= q \mathbf{x}_\alpha + J_m \mathbf{x}_\alpha. 
$$
Hence $(m^2-q) x_j=\sum_{i=1}^m x_i$ for all $1\leq j\leq m$, which again implies $x_1=x_2=\dots=x_m$. 
Similarly, one shows that $x_{m+1}=\dots=x_{2m}$, since $A^t A=qI_m+J_m$. 

Now assume $N>1$ and that we proved the claim for all $X_\cF$ of dimension less than $N$. 
Suppose $f$ is not type-constant. Then there are two vertices $x, y$ of the same type  
such that $f(x)\neq f(y)$. We claim that we can choose $x$ and $y$ so that 
there is a vertex $v\in \Ver(X_\cF)$ such that $x, y\in \Lk(v)$.  We start with $X_\cF=X_\varnothing^n$. In that case $x$ 
and $y$ correspond to subspaces $W_1$ and $W_2$ of $\F_q^{n+2}$ of the same dimension. 
By assumption $n>1$. If $\dim(W_i)=1$, then $v$ corresponding to $W_1+W_2$ is adjacent to both $x$ and $y$.  
(Note that $\dim(W_1+W_2)=2<n+2$.) If $r=\dim(W_i)>1$, choose a line $\ell_i\in W_i$. Let $W_3$ be a subspace 
of dimension $r$ which contains $\ell_1+\ell_2$. Let $z$ be the corresponding vertex. 
If $f(x)=f(z)$, then we replace $x$ by $z$ and take $v$ corresponding to $\ell_2$. If $f(x)\neq f(z)$, then we replace $y$ by $z$ 
and take $v$ corresponding to $\ell_1$. Now suppose $X_\cF\cong X_\varnothing^{n_1}\ast \cdots \ast X_\varnothing^{n_s}$, $s\geq 2$.  
Our vertices are in the same $X_\varnothing^{n_i}$ since they have the same type, but then any vertex in another 
$X_\varnothing^{n_j}$ is adjacent to both $x$ and $y$. 

Let $x, y\in \Lk(v)$ be as in the previous paragraph. Let $\Type(v)=\alpha$. 
Consider the function $\tau_vdf_\alpha$. We have 
$$
\tau_vdf_\alpha(x)=df_\alpha([v,x])=f(x)-Rf(v)\neq f(y)-Rf(v) = \tau_vdf_\alpha(y). 
$$ 
Hence $\tau_vdf_\alpha\in C^0(\Lk(v))$ is not type-constant.  
By induction, $\tau_vdf_\alpha$ does not lie in the subspace of $C^0(\Lk(v))$ spanned by eigenfunctions 
with eigenvalue $N$. This implies (use the orthonormal decomposition of Lemma \ref{lemDec5} and Proposition \ref{prop_ny})
$$
(\Delta_v \tau_vdf_\alpha, \tau_vdf_\alpha)_v< N (\tau_vdf_\alpha, \tau_vdf_\alpha)_v. 
$$
This inequality implies, as in the proof of Proposition \ref{prop_ny}, that 
$$
(\Delta \rho_\alpha df_\alpha, \rho_\alpha df_\alpha) < N (\rho_\alpha df_\alpha, df_\alpha). 
$$
As in the proof of Proposition \ref{prop_ny}, this leads to 
$$
(N+1)(N-1)c< N[(N-c)(R-1)^2+(N+1)(R^2+N)],
$$
where $c=N+1$ and $R=(N-c)/N$. But for these $c$ and $R$ both sides 
are equal, so the inequality cannot be strict. 
\end{proof}

\begin{prop}
A function $f\in C^0(X_\cF)$ is a $\Delta$-eigenfunction with eigenvalue $0$ 
if and only if $f$ is constant. A function $f\in C^0(X_\cF)$ is a $\Delta$-eigenfunction with eigenvalue $N+1$ 
if and only if $f\in \sC_0$. This implies that $M^0(X_\cF)=N+1$ and its multiplicity 
as an eigenvalue of $\Delta$ is $N$. 
\end{prop}
\begin{proof} It is easy to check that $\Delta f = 0\Leftrightarrow df=0\Leftrightarrow f$ is constant (since $X_\cF$ is connected).  
By Lemma \ref{lemNNC}, if $f\in \sC_0$, then  $\Delta f=(N+1)f$. Conversely, suppose $\Delta f=(N+1)f$.
By Lemma \ref{lemNNC2}, $f$ is type-constant, so by Lemma \ref{lemNNC}, 
$$
\widetilde{\Delta f} = \widetilde{(N+1) f} = (N+1)\tilde{f}=\Delta \tilde{f}. 
$$
This implies $f\in \sC_0$. 
\end{proof}

\begin{rem}
In \cite{PapMM}, we proved a general result about finite buildings which implies that $M^i(X_\cF)=N+1$ for all $0\leq i\leq N-1$. 
\end{rem}

\begin{prop}\label{thm-last}
$m^0(X_\cF)\leq N$.
\end{prop}
\begin{proof} 
Denote $c:=m^0(X_\cF)$ and let $f$ be a $\Delta$-eigenfunction with
eigenvalue $c$. First we claim that $c\neq N+1$. Indeed, $\Delta$ is
a semi-simple operator and if $c=N+1$ then by Proposition
\ref{prop_ny} it has only two distinct eigenvalues, namely $0$ and
$N+1$. This implies that $\Delta^2=(N+1)\Delta$. 
In $X_\cF$ we can find two vertices $x$ and $y$ which are not adjacent 
but such that there is another vertex $v$ which is adjacent to both $x$ and $y$. 
Let $g\in C^0(X_\cF)$ be a function such that $g(x)\neq 0$ but $g(x')=0$ if $x'\neq x$. 
Now $\Delta g(y)=0$ because this is a sum of the values of $g$ at $y$ and the vertices 
adjacent to $y$, and $x$ is not one of them. On the the other hand, $\Delta^2 g(y)\neq 0$ 
since this is a sum which involves $g(x)$ with a non-zero coefficient. This 
contradicts the equality $\Delta^2=(N+1)\Delta$. 

Define a function $h\in C^0(X_\cF)$ by 
$$h(v)=\sum_{\substack{x\in \Ver(X_\cF)\\ \Type(x)=\Type(v)}} f(x), \qquad \text{for any }v\in \Ver(X_\cF). 
$$
It is clear that $h$ is type-constant, and because 
$f$ is a $\Delta$-eigenfunction, we have $\Delta h=ch$. Since $c\neq 0, N+1$, 
the function $h$ must be identically $0$. Therefore,  
$\sum_{\Type(v)=\beta}f(v)=0$ for any
fixed $\beta\in \fT$. Obviously the same is also true for
$f_\alpha$, i.e., $\sum_{\Type(v)=\beta}f_\alpha(v)=0$. Since $w(v)$
depends only on the type of $v$, we see that $f_\alpha$ is orthogonal 
to $\mathbf{1}$ in $C^0(X_\cF)$ with respect to the pairing \eqref{eq-pairing}. 
As in the proof of Lemma
\ref{lemDec5}, this implies that
$$
(\Delta f_\alpha, f_\alpha)\geq c\cdot (f_\alpha, f_\alpha).
$$
Summing over all types, we get
$$
\sum_{\alpha\in \fT} (\Delta f_\alpha, f_\alpha)\geq c(N+R^2)\cdot (f, f). 
$$
Comparing this inequality with the expression in Lemma \ref{lemd31},
we conclude that $(N-c)(R-1)^2\geq 0$. Since $R$ is arbitrary, we
must have $c\leq N$.
\end{proof}

\subsection*{Acknowledgements} The author thanks Ori Parzanchevski and Farbod Shokrieh 
for useful comments on an earlier version of the paper. 

\def\polhk#1{\setbox0=\hbox{#1}{\ooalign{\hidewidth
  \lower1.5ex\hbox{`}\hidewidth\crcr\unhbox0}}}
\providecommand{\bysame}{\leavevmode\hbox to3em{\hrulefill}\thinspace}
\providecommand{\MR}{\relax\ifhmode\unskip\space\fi MR }
\providecommand{\MRhref}[2]{%
  \href{http://www.ams.org/mathscinet-getitem?mr=#1}{#2}
}
\providecommand{\href}[2]{#2}

\end{document}